\documentclass[reqno,11pt]{amsart}

\usepackage{relsize}
\usepackage{scalerel}
\usepackage{stackengine,wasysym}
\usepackage{todonotes}
\usepackage[colorlinks,
            linkcolor=blue,
            anchorcolor=blue,
            citecolor=blue
]{hyperref}

\usepackage{amsmath,amssymb}
\usepackage{mathrsfs}
\usepackage{cases}
\usepackage{xcolor}
\usepackage{verbatim}
\usepackage{cite}

\allowdisplaybreaks

\usepackage[capitalize,nameinlink,noabbrev]{cleveref}

\numberwithin{equation}{section}

\theoremstyle{plain}
\newtheorem{lemma}{Lemma}[section]
\newtheorem{theorem}[lemma]{Theorem}
\newtheorem{corollary}[lemma]{Corollary}

\theoremstyle{remark}
\newtheorem{remark}{Remark}
\theoremstyle{convention}

\def\ee{\varepsilon}
\def\aa{\alpha}
\def\bb{\beta}
\def\dd{\delta}
\def\DD{\Delta}
\def\gg{\gamma}
\def\om{\omega}
\def\ss{\sigma}

\def\ll{\lambda}
\def\LL{\Lambda}

\def\Om{\Omega} 
\def\om{\omega}
\def\pp{\partial}

\newcommand{\Cb}{\mathbb{C}}
\newcommand{\Eb}{\mathbb{E}} 
\newcommand{\Pb}{\mathbb{P}} 
\newcommand{\Rb}{\mathbb{R}}
\newcommand{\Tb}{\mathbb{T}}
\newcommand{\Zb}{\mathbb{Z}}

\newcommand{\Fc}{\mathcal{F}}
\newcommand{\Lc}{\mathcal{L}}
\newcommand{\Pc}{\mathcal{P}} 
\newcommand{\bV}{\overline{V}} 
\newcommand{\tV}{\widetilde{V}} 
\newcommand{\Vc}{\mathcal{V}} 
\newcommand{\bVc}{\overline{\mathcal{V}}} 
\newcommand{\tVc}{\widetilde{\mathcal{V}}} 
\newcommand{\bU}{\overline{U}} 
\newcommand{\tU}{\widetilde{U}} 
\newcommand{\bP}{\overline{P}} 
\newcommand{\tP}{\widetilde{P}} 

\newcommand{\mi}{\mathrm{i}} 

\newcommand{\tv}{\widetilde{v}}
\newcommand{\bv}{\overline{v}}

\newcommand{\bu}{\overline{u}}
\newcommand{\tu}{\widetilde{u}}

\newcommand{\Pscr}{\mathscr{P}}

\makeatletter
\newcommand*{\rom}[1]{\expandafter\@slowromancap\romannumeral #1@}
\makeatother

\begin{document}

\vskip 0.125in

\title[Stochastic Primitive Equations]
{Averaging principle for the stochastic primitive equations in the large rotation limit}

\date{March 2, 2025}

\author[Q. Lin]{Quyuan Lin}
\address[Q. Lin]
{	School of Mathematical and Statistical Sciences \\
Clemson University\\
Clemson, SC 29634, USA.} \email{quyuanl@clemson.edu}

\author[R. Liu]{Rongchang Liu}
\address[R. Liu]{Department of Mathematics, University of Arizona, Tucson, AZ 85721, USA.}
\email{lrc666@arizona.edu}

\author[V.R. Martinez]{Vincent R. Martinez}
\address[V.R. Martinez]
{	Department of Mathematics \& Statistics\\
    CUNY Hunter College\\
    Department of Mathematics \\
    CUNY Graduate Center  \\
	New York, NY 10065, USA.} 
\email{vrmartinez@hunter.cuny.edu}

\begin{abstract}
    It is known that the unique ergodicity of the viscous primitive equations with additive white-in-time noise remains an open problem. In this work, we demonstrate that, as the rotational intensity approaches infinity, the distribution of any given strong solution, rescaled by the rotation, is attracted to the unique invariant measure of the stochastic limit resonant system. This suggests that the system exhibits nearly unique ergodicity in the large rotation limit. The proof is based on a stochastic averaging principle combined with a coupling argument.
\end{abstract}

\maketitle


\vskip0.1in


\section{Introduction}\label{sect:intro}
The system of three-dimensional primitive equations (PE) is a crucial model for studying large-scale oceanic and atmospheric dynamics, where the vertical scale is significantly smaller than the horizontal scale. It can be derived as the asymptotic limit of the small aspect ratio from the Boussinesq system or Navier-Stokes equations, which have been justified rigorously in \cite{azerad2001mathematical,li2019primitive,li2022primitive}.  In this work we consider the following \textit{non-dimensional} PE with additive noise over the unit torus $\Tb^3 = \Rb^3/\Zb^3$:
\begin{subequations}\label{PE-system}
\begin{align}
    &\partial_t v + v\cdot \nabla_h  v + w\pp_z v  + \aa v^\perp + \nabla_h  P   = \nu \Delta v+\sigma \partial_tW, 
    \\
    &\pp_z P   = 0,
    \\
    &\nabla_h  \cdot v + \pp_z w  = 0.
\end{align}
\end{subequations}
Here the horizontal velocity field $v=(v_1, v_2)$, the vertical velocity $w$, and the pressure $P$ are the unknowns. We denote by $\nabla_h  = (\partial_x, \partial_y)$ and $\nabla= (\partial_x, \partial_y, \partial_z)$ the 2D horizontal and full 3D gradients. The positive constant $\nu$ is the viscosity, and $ \alpha v^\perp = \alpha (-v_2, v_1)$ represents the Coriolis force with parameter $ \alpha>0$. Here $W$ is a cylindrical Wiener process and $\sigma\partial_tW$ is  the additive noise. Note that we have omitted the coupling with temperature in \eqref{PE-system} for simplicity and clarity. In this case, \eqref{PE-system} is also referred to as the hydrostatic Navier-Stokes equations.

These equations have been extensively studied over the past few decades. 
In the deterministic setting, the existence of weak solutions was established in \cite{lions1992new,lions1992equations}, although uniqueness remains an open problem. The global existence and uniqueness of strong solutions have been demonstrated in \cite{cao2007global,kobelkov2006existence,kukavica2007regularity,hieber2016global}, with further discussion in \cite{li2016recent,hieber2020approach}. The well-posedness results have been extended to the stochastic system with additive noise, i.e., system \eqref{PE-system}, in \cite{guo20093d}, where the existence of a random pull-back attractor is obtained.

We are interested in the ergodic behavior of the solutions to system \eqref{PE-system}. When the noise is bounded, unique ergodicity and mixing have been shown in \cite{boulvard2023controllability,boulvard2022mixing}. However, for the white noise perturbation in system \eqref{PE-system}, unique ergodicity remains an open problem due to the weak tail estimates of the solutions. Currently, only logarithmic moment bounds are available, as shown in \cite{glatt2014existence}, where the existence and smoothness of invariant measures have been established.  We refer the reader to \cite{kuksin2012mathematics} for more information on the ergodic theory of SPDEs, which has been widely studied over the years.   

In this work, we prove that, in the limit of fast rotation, system \eqref{PE-system} is close to being uniquely ergodic. More precisely, let \( P_{t}^{\alpha}(v,\cdot)\) be the transition probabilities of the strong solution $v \in H^1$ to system \eqref{PE-system}, and let \(\mathcal{U}_t\) be defined by
\begin{align}\label{e.121701}
    \mathcal{U}_t(\bv, \tv) = (\bv, e^{\alpha Jt} \tv), \quad J=\left(\begin{array}{cc}0 & -1 \\ 1 & 0\end{array}\right), 
\end{align}
for any \(v = (\bv, \tv) \in H^1\), where \(\bv = \bP v\) (see \eqref{e.121201} below) is the barotropic mode, and \(\tv = (I - \bP) v\) is the baroclinic mode. In the limit of fast rotation, we derive the following limiting system:
\begin{align}\label{e.080602_intro}
    \begin{split}
        &\partial_t\bV+\Pc_h(\bV\cdot\nabla_h  \bV) = \nu\DD_h \bV + \Pc_h\overline{\sigma} \partial_t\widetilde{W},\\
        &\partial_t\tV +\bV\cdot\nabla_h \tV+\frac12\tV^{\perp}\left(\nabla_h ^{\perp}\cdot\bV\right)=\nu\Delta \tV+ \sqrt{\widetilde{Q}}\partial_t\widetilde{W}, 
    \end{split}
\end{align}
where $\Pc_h$ is the 2D Leray projection, $\widetilde{W}$ is a cylindrical Wiener process defined on a possibly different probability space than $W$, and $\widetilde{Q}$ is an ``averaged" covariance operator (defined precisely in \eqref{def:Q_avg}). Let \(\Pscr(H^1)\) be the space of probability measures on \(H^1\). Then an informal version of our main result can be stated as follows:

\noindent {\bf Main Theorem}. \label{mainthm}{\it Under mild non-degenerate assumptions on the noise, there exists a unique probability measure $\mu$ such that for any $v\in H^2$, we have
\[
    \mathcal{U}_t^* P_{t}^{\alpha}(v, \cdot) \to \mu
\]
weakly in \(\Pscr(H^1)\) as \(\alpha \to \infty\) and \(t \to \infty\), where $\mathcal{U}_t^* P_{t}^{\alpha}(v, \cdot)$ denotes the pushforward of the measure \(P_{t}^{\alpha}(v, \cdot)\) by \(\mathcal{U}_t\). Moreover, $\mu$ is the unique invariant probability measure corresponding to the limit resonant system \eqref{e.080602_intro}.}

{Several recent results in the literature \cite{Nguyen2023, Glatt-HoltzMondaini2024, Glatt-HoltzMartinezNguyen2024, Nguyen2024} have developed frameworks in various scenarios under which the dynamics of a stochastic system, $P_t$, and its perturbation, $P^\varepsilon_t$, ``commute" in the sense that $\lim_{t\to\infty}\lim_{\varepsilon\to0^+}P^\varepsilon_t=\lim_{\varepsilon\to0^+}\lim_{t\to\infty}P^\varepsilon_t$. Within the context of those frameworks, such results are afforded by the availability of geometric ergodicity for both the original system and its perturbation, in addition to suitable exponential moment bounds that are uniform in the singular parameter, $\varepsilon$. From this point of view, informally speaking, the Main Theorem establishes the order of limits $\lim_{t\to\infty}\lim_{\varepsilon\to0^+}P^\varepsilon_t=\mu$, where $\mu$ is the unique invariant probability measure of the system obtained in the singular limit $\varepsilon\to0^+$. It remains an outstanding open problem to establish unique ergodicity for \eqref{PE-system}, \textit{even for sufficiently large, but finite values, of the rotation rate}, which is necessary to address the issue of passing to the limit in the opposite order.  While it is conceivable that a large class of ``physically constructed" invariant measures for \eqref{PE-system}, e.g., those constructed by a Bogoliubov-Krylov-type time-averaging procedure, converge to $\mu$, such a result appears to be precluded by the lack of access to better moment bounds. Indeed, in the context of \eqref{PE-system}, the best-known bounds, obtained in \cite{glatt2014existence}, are \textit{triple logarithmic} (see \cref{l.080902}).  Nevertheless, our result is able to successfully expand upon the stochastic averaging principle framework developed for the three-dimensional rotating Navier-Stokes equation in \cite{flandoli2012stochastic} to show that (upon rescaling by \( \mathcal{U}_t \)) the distribution of \textit{any strong solution} to the system \eqref{PE-system} is attracted by the unique probability measure \( \mu \) as \(\alpha \to \infty\) and \(t \to \infty\).
}

In the existing literature, it is well-known that rapid rotation induces strong dispersion and an averaging mechanism that weakens nonlinear effects in fluid systems. For deterministic systems, such a mechanism allows for establishing the global well-posedness in the Navier-Stokes case and prolongs the lifespan of the solution in the Euler case \cite{babin1997regularity,babin1999global,babin1999regularity} (see also \cite{chemin2006mathematical,dutrifoy2005examples,embid1996averaging,koh2014strichartz} and references therein). Rotation also plays a significant role in the dynamics of the deterministic PE. For analytic initial data, it has been shown that the lifespan of solutions to the 3D inviscid PE or the 3D PE with only vertical viscosity can be extended if the rotation is sufficiently fast \cite{ghoul2022effect,lin2022effect}. The implementation of this mechanism in SPDEs was first explored in \cite{flandoli2012stochastic}, where global well-posedness up to exceptional events was established for the 3D rotating Navier-Stokes equations with additive noise, given sufficiently high rotation intensity. For the same system, with no-slip boundary conditions and well-prepared initial data, \cite{wang2024zero} studied  the convergence of solutions in the limit of infinite rotation and vanishing vertical viscosity.  To the best of our knowledge, our work seems to be the first to examine the effects of fast rotation on the ergodic behavior of SPDEs.

The overall strategy for proving the main theorem consists of two major steps. First, we show that the rescaled solutions of system \eqref{e.080602} converge in law to solutions of the limit resonant system \eqref{e.080602_intro}, based on a stochastic averaging principle inspired from \cite{flandoli2012stochastic}. Next, using coupling techniques, we demonstrate that the limit resonant system \eqref{e.080602_intro} possesses a unique invariant measure that is exponentially mixing. In the following discussion, we address key challenges in the analysis that arise from the special structure of the PE system.

For the stochastic Navier-Stokes equations \cite{flandoli2012stochastic}, the analysis of the resonant operator arising from the nonlinear terms follows similarly from the deterministic results \cite{babin1997regularity,babin1999global,babin1999regularity}. These results are based on analyzing the action of the Poincar\'e-Coriolis rotation operator on the Fourier-curl basis, which reveals the existence of frequency selection resonances. In contrast, no such selection occurs in the PE, where the resonant operator exhibits resonance across all frequencies. This phenomenon was first identified in \cite{ghoul2022effect,lin2022effect}, which rely on the barotropic-baroclinic decomposition of the system, followed by a further decomposition of the baroclinic mode along the eigenbasis of the rotation matrix. The methods from \cite{ghoul2022effect,lin2022effect} clearly apply to our viscous system \eqref{PE-system} as they share the same nonlinear terms, however, we adopt a different approach, which may be of independent interest. Specifically, we exploit the isomorphism between the complex plane $\mathbb{C}$ and $\mathbb{R}^2$, interpreting the action of the rotation matrix on $\mathbb{R}^2$ as a multiplication operator on $\mathbb{C}$. This coordinate-free approach significantly simplifies the relevant computations involving the differential operators. 

The analysis of resonance in the stochastic integral shares a similar spirit with that in \cite{flandoli2012stochastic}, but it exhibits distinct behavior. As shown in \cite{flandoli2012stochastic}, frequency selection resonance exists in the averaged covariance operator for the stochastic Navier-Stokes equations. For the PE system \eqref{PE-system}, the barotropic part of the stochastic integral becomes entirely resonant without oscillation, while the baroclinic component is highly oscillatory, with all frequencies resonating in the averaged covariance operator.

To address the convergence of oscillatory terms as the rotation intensity $\alpha \to \infty$, a meticulous analysis of relevant moment estimates is necessary to apply the abstract stochastic averaging results in \cite{flandoli2012stochastic}, primarily due to the loss of horizontal derivatives in the vertical velocity $w$, which is a special feature of the PE system and one of the major obstacles to proving unique ergodicity. As part of this analysis, we employ the method from \cite{glatt2014existence} to derive logarithmic moment bounds for higher-order Sobolev norms. Although these bounds are insufficient to establish unique ergodicity, they turn out to be sufficient to prove the convergence of the solution processes in probability, which is adequate for our purposes.

The limit resonant system consists of a barotropic part and a baroclinic part, where the former is the 2D stochastic Navier-Stokes equations, and the latter is a 3D stochastic linear heat equation, transported and stretched by the barotropic part. We emphasize that the natural phase space for the PE is $H^1$ rather than $L^2$, since the uniqueness of weak solutions is still an open problem. Therefore, to prove the main result, we show that the limit resonant system is exponentially mixing in $H^1$, where there is no cancellation of nonlinear terms. To overcome this issue, we first demonstrate exponential mixing in $ L^2 $ through a coupling argument \cite{butkovsky2020generalized}, and then enhance the mixing to $ H^1 $ using the parabolic smoothing property of the system,  drawing on similar ideas from \cite{kuksin2012mathematics}. To address the transport and stretching terms in the 3D stochastic heat equation, we perform careful anisotropic estimates to leverage the 2D nature of the barotropic part.

The paper is organized as follows. In \cref{s.121501}, we provide the mathematical preliminaries and settings, and state our main result more precisely. \cref{s.121502} is devoted to deriving the limit resonant system. In \cref{s.121503}, we prove the convergence in law of the rescaled solutions of the PE system \eqref{PE-system} to that of the limit resonant system. The exponential mixing of the limit resonant system is shown in \cref{s.121504}. Finally, the proof of our main result is given in \cref{sect:final}.

\section{Mathematical Preliminaries}\label{s.121501}
Let $\Tb^3 = \Rb^3/\Zb^3$ be the unit torus,  $\LL = \sqrt{-\Delta}$, and $\bP$ be the projection operator defined as 
\begin{align}\label{e.121201}
    \bP f = \int_0^1f(x,y,z)dz.
\end{align}
We can write $f=\overline{f}+\widetilde{f}$, where $\overline{f}:={\bP}f$ is called the \textit{barotropic mode} and $\widetilde{f}:=(I-\bP)f$ is called the \textit{baroclinic mode}. Let 
\[
    H = \left\{f\in L^2(\Tb^3,\Rb^2): f \text{ has zero mean over $\Tb^3$ and is even in $z$}\right\}.
\]
We also denote $H^s = W^{s,2}\cap H$ and $\|\cdot\|_{H^s}$ the corresponding norm, where $W^{s,2}$ is the usual Sobolev space of order $s\in \Rb$. {Note that as $f\in H$ has zero mean, the homogeneous Sobolev norm $\|f\|_{\dot{W}^{s,2}}=: \|f\|_s$ is equivalent to $\|f\|_{H^s}$.} The space of Hilbert-Schmidt operators from $H\to H^s$ is denoted by $L_2(H,H^s)$ with norm 
\begin{align*}
    \|A\|_{L_2(H,H^s)}: =\left(\sum_{j=1}^{\infty}\|A \boldsymbol{e}_j\|_{H^s}^2\right)^{1/2}, \, A\in L_2(H,H^s),
\end{align*}
where $\{\boldsymbol{e}_j\}_{j\in \mathbb N}$ is a complete orthonormal basis in $H$.  
We assume that $\sigma$ is a self-adjoint bounded linear operator on $H$ such that $\ss\in L_2(H,H^2)$ and $\bP\ss=\ss \bP$. Besides, 
$W = \sum_{k=1}^{\infty}W_j\boldsymbol{e}_j$
is the cylindrical Wiener process in $H$, where $W_k$ is a sequence of independent Brownian motions. 

The vertical velocity $w$ is determined by the horizontal velocity $v$ through the relation
\begin{align}\label{e.081104}
    w(v)(x,y,z) = -\int_0^z\nabla_h \cdot v(x,y,\zeta)d\zeta.
\end{align}
In addition, let 
$J=\left(\begin{array}{cc}0 & -1 \\ 1 & 0\end{array}\right)$
be the rotation matrix. It is clear that for all $\alpha \in \mathbb R$, $t\geq 0$, and $s\in \mathbb R$, one has
$
    \|e^{\aa J t}u\|_{s} = \|u\|_{s}
$
for any $u\in H^s$.

Applying the notations above, 
the PE can then be written in the barotropic-baroclinic form as 
\begin{subequations}\label{e.080605}
    \begin{align}
    & \partial_t\bv  + \Pc_h \Big(\bv\cdot \nabla_h  \bv + \bP \Big((\nabla_h \cdot \tv) \tv + \tv\cdot \nabla_h  \tv \Big) \Big) = \nu\DD_h \bv + \Pc_h\overline{\sigma}\partial_tW, \label{e.080603} \\
    & \partial_t \tv + \tv \cdot \nabla_h  \tv + \tv \cdot \nabla_h  \bv + \bv \cdot \nabla_h  \tv+w(\tv) \partial_z \tv - \bP\Big(\tv \cdot \nabla_h  \tv + (\nabla_h  \cdot \tv) \tv \Big)+ \aa J\tv \nonumber
    \\
    &\hspace{3cm}  = \nu\Delta \tv+\widetilde{\sigma}\partial_tW . \label{e.080604}
    \end{align}
\end{subequations}
Here $\Pc_h$ denotes the 2D Leray projection. It is well-known that the system \eqref{e.080605} is globally well-posed in $H^1$ \cite{guo20093d} with pathwise solutions, $v=(\bv,\tv)$, satisfying $v\in C([0,T];H^1)\cap L^2(0,T;H^2)$ almost surely for any $T>0$. The solution $v$ depends on $\aa$ and, when necessary, we will express this dependence by $v^{\aa}$.

To state the non-degeneracy condition on the noise in our main result, we let 
\begin{align}\label{e.011501}
    \{(\ll_k, e_k)\}_{k\in K},\quad K:=2\pi\mathbb Z^3\setminus\{(0,0,0)\}
\end{align}
be the corresponding eigenvalues and real eigenfunctions of $-\Delta$ on $L^2(\Tb^3,\Rb)$. We denote by $\ll_{|k|}:=\ll_k$ for $k\in K$.
We distinguish between barotropic and baroclinic wave-numbers by
\[
    \overline{K}=\left\{(k_1,k_2,k_3)\in K :k_3=0\right\}, \quad \widetilde{K} = K\setminus \overline{K}.
\]
For any integer $N\geq 1$, we define the band-limited barotropic and baroclinic projections by
\begin{align*}
    \bP_N& : L^2(\Tb^3,\Rb)\to \mathrm{span}\left\{e_k: \ll_k\leq \ll_N,\, k\in \overline{K}\right\},\\
    \tP_N&: L^2(\Tb^3,\Rb)\to \mathrm{span}\left\{e_k: \ll_k\leq \ll_N,\, k\in \widetilde{K}\right\},
\end{align*}
and denote by $P_N$ the projection to the modes with wave-number $k$ such that $\ll_k\leq \ll_N$. {Note that these projections, when applied to $H$, are understood as acting component-wise.} Note that $P_N=\bP_N+\tP_N$. 
Lastly, we define the averaged covariance operator of the baroclinic component of the noise by
\[
    \widetilde{Q} = \lim_{T\to\infty}\frac{1}{T}\int_0^{T}e^{J\tau}(I-\bP)\sigma^2(I-\bP)e^{-J\tau} d\tau. 
\]
Recall $\mathcal{U}_t$ defined as in \eqref{e.121701}. We may now state a formal version of the \textit{first claim} of main result (as stated in \cref{sect:intro}) of this article:

\begin{theorem}\label{t.121702}
    Given $\ss\in L_2(H,H^2)$, there exists $N>0$, independent of $\alpha$, such that if
    \begin{align*}
        \bP_N H\subset \mathrm{Range}(\bP\Pc_h\ss)\quad\text{and}\quad \tP_N H\subset \mathrm{Range}\left(\widetilde{Q}^{1/2}\right),
    \end{align*}
    then there exists a unique probability measure, $\mu\in\Pscr(H^1)$, such that for any \(v \in H^2\), 
    \[
        \mathcal{U}_t^* P_{t}^{\alpha}(v, \cdot) \to \mu
    \]
    weakly in \(\Pscr(H^1)\) as \(\alpha \to \infty\) and \(t \to \infty\). 
\end{theorem}

The second claim of the main result is that $\mu$ is the unique invariant probability measure corresponding to the limit resonant system. In order to make this identification, we will make use of geometric ergodicity of the limit resonant system (under appropriate conditions on the noise); this step will be executed in \cref{s.121504}. Before doing so, we will develop a formal derivation of the limit resonant system in \cref{s.121502}.

\subsection*{Convention}
The universal constant $C$ appearing in the paper may change from line to line. We will emphasize the dependence of $C$ on some parameters when necessary.

\section{The limit resonant system}\label{s.121502}
In this section, we derive the limit resonant system of the PE obtained as $\aa\to \infty$. Note that there is no rotation acting directly on the barotropic component. Introducing the new variable $\tu(t) = e^{\aa Jt}\tv(t)$, then by equation \eqref{e.080604}, the new variable satisfies
\begin{align}\label{e.081001}
    \begin{split}
        \partial_t\tu &= \aa J e^{\aa Jt}\tv + e^{\aa Jt}\partial_{t}\tv\\
        & = \aa J e^{\aa Jt}\tv - e^{\aa Jt}\left(\tv \cdot \nabla_h  \tv + \tv \cdot \nabla_h  \bv + \bv \cdot \nabla_h  \tv+w(\tv) \partial_z \tv - \bP\Big(\tv \cdot \nabla_h  \tv + (\nabla_h  \cdot \tv) \tv \Big)\right)\\
        & \qquad - e^{\aa Jt}\aa\tv^{\perp}+\nu\Delta e^{\aa Jt}\tv + e^{\aa Jt}\widetilde{\sigma}\partial_tW\\
        &= - e^{\aa Jt}\left(\tv \cdot \nabla_h  \tv + \tv \cdot \nabla_h  \bv + \bv \cdot \nabla_h  \tv+w(\tv) \partial_z \tv - \bP\Big(\tv \cdot \nabla_h  \tv + (\nabla_h  \cdot \tv) \tv \Big)\right)\\
        &\qquad +\nu\Delta\tu + e^{\aa Jt}\widetilde{\sigma}\partial_tW.
    \end{split}
\end{align}
Note that by \cref{l.081002}, 
\begin{align*}
    e^{\aa Jt}\left(\tv \cdot \nabla_h  \bv + \bv \cdot \nabla_h  \tv\right) &= e^{\aa Jt}\left(\left(e^{-\aa Jt}\tu\right) \cdot \nabla_h  \bv + \bv \cdot \nabla_h  \left(e^{-\aa Jt}\tu\right)\right)\\
    & = \left(e^{-\aa Jt}\tu\right) \cdot \nabla_h  \left(e^{\aa Jt}\bv\right) + \bv \cdot \nabla_h  \tu\\
    & = \frac12\left(\tu\cdot \nabla_h  \bv -\tu^{\perp}\cdot \nabla_h \bv^{\perp} \right) + \frac12e^{2\alpha Jt}\left(\tu\cdot \nabla_h  \bv -\tu\cdot \nabla_h ^{\perp}\bv^{\perp} \right)+ \bv \cdot \nabla_h  \tu\\
    & = \frac12\tu^{\perp}\left(\nabla_h ^{\perp}\cdot \bv\right) + \frac12e^{2\alpha Jt}\left(\tu\cdot \nabla_h  \bv -\tu\cdot \nabla_h ^{\perp}\bv^{\perp} \right)+ \bv \cdot \nabla_h  \tu, 
\end{align*}
where we used 
\[\tu\cdot \nabla_h  \bv -\tu^{\perp}\cdot \nabla_h \bv^{\perp} = \tu^{\perp}\left(\nabla_h ^{\perp}\cdot \bv\right)\] 
due to $\nabla_h \cdot \bv =0$. 
Next, we apply \cref{l.081002} again and obtain that 
\begin{align}\label{e.080607}
    \begin{split}
        \tv \cdot \nabla_h  \tv& = \left(e^{-\aa Jt}\tu\right) \cdot \nabla_h \left(e^{-\aa Jt}\tu\right)\\
        & = \frac12e^{-2\aa Jt}\left(\tu\cdot \nabla_h  \tu -\tu^{\perp}\cdot \nabla_h \tu^{\perp}\right) + \frac12 \left(\tu\cdot \nabla_h  \tu -\tu\cdot \nabla_h ^{\perp}\tu^{\perp}\right)\\
        & = \frac12e^{-2\aa Jt}\left(\tu\cdot \nabla_h  \tu -\tu^{\perp}\cdot \nabla_h \tu^{\perp}\right)+\frac14 \nabla_h |\tu|^2 + \frac12 J\left(\nabla_h  \tu\right)^{T}\tu^{\perp},\\
    \end{split}
\end{align}
and 
\begin{align}\label{e.080608}
    \begin{split}
        \left(\nabla_h \cdot\tv\right)\tv& =\left(\nabla_h  \cdot  e^{-\aa Jt}\tu\right)\left(e^{-\aa Jt}\tu\right) \\
        & =\frac12e^{-2\aa Jt}\left((\nabla_h \cdot\tu) \tu - (\nabla_h \cdot \tu^\perp) \tu^\perp\right)+ \frac12\left(\left(\nabla_h \cdot \tu\right)\tu-\left(\nabla_h ^{\perp}\cdot \tu\right)\tu^{\perp}\right)\\
        & =\frac12e^{-2\aa Jt}\left(\tu\cdot \nabla_h  \tu - \tu^{\perp}\cdot \nabla_h \tu^{\perp}\right)+\frac14 \nabla_h |\tu|^2 - \frac12 J\left(\nabla_h  \tu\right)^{T}\tu^{\perp}.
    \end{split}
\end{align}
Consequently, 
\begin{align}\label{e.080609}
    \begin{split}
        \tv \cdot \nabla_h  \tv + (\nabla_h  \cdot \tv) \tv = e^{-2\aa Jt}\left(\tu\cdot \nabla_h  \tu - \tu^{\perp}\cdot \nabla_h \tu^{\perp}\right)+\frac12\nabla_h |\tu|^2. 
    \end{split}
\end{align}
Since $e^{\aa Jt}$ commutes with $\bP$, we obtain from \eqref{e.080607} and \eqref{e.080609} that 
\begin{align*}
    &e^{\aa Jt}\left( \tv \cdot \nabla_h  \tv- \bP\Big(\tv \cdot \nabla_h  \tv + (\nabla_h  \cdot \tv) \tv \Big)\right)\\
    & = \frac12e^{-\aa Jt}\left(\tu\cdot \nabla_h  \tu -\tu^{\perp}\cdot \nabla_h \tu^{\perp}\right) + \frac12 e^{\aa Jt}\left(\tu\cdot \nabla_h  \tu -\tu\cdot \nabla_h ^{\perp}\tu^{\perp}-\bP\nabla_h |\tu|^2\right)\\ 
    &\quad - e^{-\aa Jt}\bP\left(\tu\cdot \nabla_h  \tu - \tu^{\perp}\cdot \nabla_h \tu^{\perp}\right)\\
    &=\frac12e^{-\aa Jt}(I-2{\bP})\left(\tu\cdot \nabla_h  \tu -\tu^{\perp}\cdot \nabla_h \tu^{\perp}\right) + \frac12 e^{\aa Jt}\left(\tu\cdot \nabla_h  \tu -\tu\cdot \nabla_h ^{\perp}\tu^{\perp}-\bP\nabla_h |\tu|^2\right).  
\end{align*}
We also have 
\begin{align*}
    e^{\aa Jt}w(\tv) \partial_z \tv = -\int_0^z\nabla_h  \cdot \tv(x,y,\zeta)d\zeta\partial_{z}\tu = w(e^{-\aa Jt}\tu)\partial_{z}\tu.
\end{align*}
Therefore, we obtain the equation for $\tu$: 
\begin{align}\label{e.080606}
    \begin{split}
        \partial_t\tu 
        &= -\Bigg(\bv \cdot \nabla_h  \tu+ \frac12\tu^{\perp}\left(\nabla_h ^{\perp}\cdot \bv\right) + \frac12e^{2\alpha Jt}\left(\tu\cdot \nabla_h  \bv -\tu\cdot \nabla_h ^{\perp}\bv^{\perp} \right) \\
        &\quad +\frac12e^{-\aa Jt}(I-2{\bP})\left(\tu\cdot \nabla_h  \tu -\tu^{\perp}\cdot \nabla_h \tu^{\perp}\right) + \frac12 e^{\aa Jt}\left(\tu\cdot \nabla_h  \tu -\tu\cdot \nabla_h ^{\perp}\tu^{\perp}-\bP\nabla_h |\tu|^2\right)  \\
        &\quad +w(e^{-\aa Jt}\tu)\partial_{z}\tu\Bigg)+\nu\Delta\tu + e^{\aa Jt}\widetilde{\sigma}\partial_tW.
    \end{split}
\end{align}
From \eqref{e.080609} and the fact $\Pc_h{\bP}(\nabla_h |\tu|^2)=0$ we also obtain the equation of $\bv$: 
\begin{align}
    \begin{split}\label{e.080610}
        \partial_t\bv  + \Pc_h \Big(\bv\cdot \nabla_h  \bv + e^{-2\aa Jt}\bP\left(\tu\cdot \nabla_h  \tu - \tu^{\perp}\cdot \nabla_h \tu^{\perp}\right) \Big) = \nu\DD_h \bv + \Pc_h\overline{\sigma}\partial_tW.
    \end{split}
\end{align}
The system consisting of the two equations above for $(\bv, \tu)$ is equivalent to the original PE system \eqref{e.080605} with the same initial data through the change of variable 
\begin{align*}
    (\bv, \tu) = (\bv, e^{\aa Jt}\tv). 
\end{align*}

There are then two cases for the limit resonant system.  Firstly, if there is no noise acting on the baroclinic component, then the limit resonant system is 
\begin{align}\label{e.080601}
    \begin{split}
        &\partial_t\bV+\Pc_h(\bV\cdot\nabla_h  \bV) = \nu\DD_h \bV + \Pc_h\overline{\sigma} \partial_tW,\\
        &\partial_t\tV +\bV\cdot\nabla_h \tV+\frac12\tV^{\perp}\left(\nabla_h ^{\perp}\cdot\bV\right)=\nu\Delta \tV. 
    \end{split}
\end{align}

Secondly, as shown above, if the baroclinic component is also driven by the noise ($\widetilde{\sigma}\neq 0$), then we need to deal with the limit of the stochastic integration 
\[M_{\alpha}(t) := \int_0^te^{\aa Jr}\widetilde{\sigma}dW(r).\] 
As in \cite{flandoli2012stochastic}, we will show that this Gaussian martingale will converge in law to a Wiener process in $H$ with covariance 
    \begin{align}\label{def:Q_avg}
        \widetilde{Q} = \lim_{T\to\infty}\frac{1}{T}\int_0^{T}e^{J\tau}(I-\bP)\sigma^2(I-\bP)e^{-J\tau} d\tau. 
    \end{align}
Therefore the limit resonant system in this case is 
\begin{align}
    \begin{split}
        &\partial_t\bV+\Pc_h(\bV\cdot\nabla_h  \bV) = \nu\DD_h \bV + \Pc_h\overline{\sigma} \partial_t\widetilde{W},\\
        &\partial_t\tV +\bV\cdot\nabla_h \tV+\frac12\tV^{\perp}\left(\nabla_h ^{\perp}\cdot\bV\right)=\nu\Delta \tV+ \sqrt{\widetilde{Q}}\partial_t\widetilde{W}, 
    \end{split}
\end{align}
where $\widetilde{W}$ is a cylindrical Wiener process in $H$ defined on a new probability space $(\widetilde{\Om},\widetilde{\Fc},\widetilde{\Fc}_t,\widetilde{\Pb})$.

Summarizing the above discussions, we have the following system of $(\bv,\tu)$ that is equivalent to the original system \eqref{e.080605} of $(\bv,\tv)$ through $ (\bv, \tu) = (\bv, e^{\aa Jt}\tv)$:
\begin{align}\label{e.080901}
    \begin{split}
        \partial_t\bv &= -\Pc_h \Big(\bv\cdot \nabla_h  \bv + e^{-2\aa Jt}\bP\left(\tu\cdot \nabla_h  \tu - \tu^{\perp}\cdot \nabla_h \tu^{\perp}\right) \Big) = \nu\DD_h \bv + \Pc_h\overline{\sigma}\partial_tW,\\
        \partial_t\tu 
        &= -\Bigg(\bv \cdot \nabla_h  \tu+ \frac12\tu^{\perp}\left(\nabla_h ^{\perp}\cdot \bv\right) + \frac12e^{2\alpha Jt}\left(\tu\cdot \nabla_h  \bv -\tu\cdot \nabla_h ^{\perp}\bv^{\perp} \right) \\
        &\quad +\frac12e^{-\aa Jt}(I-2{\bP})\left(\tu\cdot \nabla_h  \tu -\tu^{\perp}\cdot \nabla_h \tu^{\perp}\right) + \frac12 e^{\aa Jt}\left(\tu\cdot \nabla_h  \tu -\tu\cdot \nabla_h ^{\perp}\tu^{\perp}+\bP\nabla_h |\tu|^2\right)  \\
        &\quad +w(e^{-\aa Jt}\tu)\partial_{z}\tu\Bigg)+\nu\Delta\tu + e^{\aa Jt}\widetilde{\sigma}\partial_tW.
    \end{split}
\end{align}
The limit resonant system is 
\begin{align}\label{e.080602}
    \begin{split}
        &\partial_t\bV+\Pc_h(\bV\cdot\nabla_h  \bV) = \nu\DD_h \bV + \Pc_h\overline{\sigma} \partial_t\widetilde{W},\\
        &\partial_t\tV +\bV\cdot\nabla_h \tV+\frac12\tV^{\perp}\left(\nabla_h ^{\perp}\cdot\bV\right)=\nu\Delta \tV+ \sqrt{\widetilde{Q}}\partial_t\widetilde{W}, 
    \end{split}
\end{align}
where $\widetilde{W}$ is a cylindrical Wiener process in $H$ defined on a new probability space $(\widetilde{\Om},\widetilde{\Fc},\widetilde{\Fc}_t,\widetilde{\Pb})$, and the covariance, $\widetilde{Q}$, is given by \eqref{def:Q_avg}. 

\begin{remark}
    {To see the operator $\widetilde{Q}$ more clearly, let us expand the functions and operators with respect to the complete orthonormal system given by 
    \begin{align*}
        \{\phi_{\gg}e^{\mi k\cdot x}\}_{k\in2\pi\Zb^3\setminus\{(0,0,0)\},\gg=1,2},
    \end{align*}
    where $\phi_{\gg}$ is the (normalized) eigenvector of $J$ corresponding to the eigenvalue $(-1)^{\gg} \mi$. Assume that under the basis $\{\phi_{\gg}\}$, the matrix representation of $\ss$ is 
    \begin{align*}
        \ss = \begin{bmatrix}
            \ss_{11} & \ss_{12} \\
            \ss_{21} & \ss_{22}
        \end{bmatrix},
    \end{align*}
    which is equivalent to say that 
    \begin{align*}
        &\ss\phi_{\gg} = \ss_{\gg1}\phi_1+\ss_{\gg2}\phi_2, 
    \end{align*} 
    Since $\ss$ is self-adjoint, we have 
    \begin{align*}
        \ss^2 = \ss\ss^{*} = \begin{bmatrix}
            \ss_{11}\ss_{11}^*+\ss_{12}\ss_{12}^* & \ss_{11}\ss_{21}^*+\ss_{12}\ss_{22}^* \\
            \ss_{21}\ss_{11}^*+\ss_{22}\ss_{12}^* & \ss_{21}\ss_{21}^*+\ss_{22}\ss_{22}^*
        \end{bmatrix}.
    \end{align*}
    Denote the square root of the diagonal part of $\ss^2$ as 
    \begin{align*}
        \Upsilon=\begin{bmatrix}
            \ss_{11}\ss_{11}^*+\ss_{12}\ss_{12}^* & 0 \\
            0 & \ss_{21}\ss_{21}^*+\ss_{22}\ss_{22}^*
        \end{bmatrix}^{\frac12}.
    \end{align*}

    Then a straightforward calculation gives $\widetilde{Q} = (I-\bP)\Upsilon^2(I-\bP)$. So the range of $\sqrt{\widetilde{Q} }$ then behaves like the range of $\Upsilon$ up to the projection $I-\bP$.}
\end{remark}

\section{Convergence to the Limit Resonant System on Finite-Time Horizons}\label{s.121503}
In this section, we prove the convergence \textit{in law} of the solution of \eqref{e.080901} to that of \eqref{e.080602}, {recalling that \eqref{e.080901} is equivalent to \eqref{e.080605}, on arbitrary, but finite-length time intervals $[0,T]$. This will proceed in two steps. First, we prove the convergence \textit{in probability} of the solution $v^{\aa}$ to the solution $V^{\aa}$ of an auxiliary equation, which is the limit resonant system modified to have the same noise process as the perturbed system \eqref{e.080901}:
\begin{align}\label{e.080803}
    \begin{split}
        &\partial_t\bV^\alpha+\Pc_h(\bV^\alpha\cdot\nabla_h  \bV^\alpha) = \nu\DD_h \bV^\alpha + \Pc_h\overline{\sigma} \partial_tW,\\
        &\partial_t\tV^\alpha +\bV^\alpha\cdot\nabla_h \tV^\alpha+\frac12(\tV^{\alpha})^{\perp}\left(\nabla_h ^{\perp}\cdot\bV^\alpha\right)=\nu\Delta \tV^\alpha+ e^{\aa Jt}\widetilde{\sigma}\partial_tW. 
    \end{split}
\end{align}
Second, we show that $V^{\aa}$ converges in law to $V$, the solution to \eqref{e.080602}, where both systems have the same drift terms, but correspond to different noise processes.
In particular, the main result of this section is obtained as a corollary of these two intermediate convergence results.

{More precisely, we first prove:}

\begin{theorem}\label{thm:convergence_A}
    Suppose $v^{\alpha} = (\bv,\tu)$ is the solution to \eqref{e.080901} with deterministic initial data $v_0\in H^2$ and that $V^{\aa} = (\bV^{\aa}, \tV^{\aa})$ is the solution to \eqref{e.080803} corresponding to the same initial data. Then for any $T>0$ and $\dd>0$, there is $\alpha_0=\alpha_0(T,\dd)>0$ such that for all $\alpha>\alpha_0$, 
    \begin{align*}
        \Pb\left(\sup_{0\leq t\leq T}\|v^{\alpha}(t)-V^{\aa}(t)\|_{H^1}\leq \dd\right)\geq 1-\dd. 
    \end{align*}
\end{theorem}

\begin{remark}
    If $\widetilde{\sigma}=0$, then $V^{\aa}$ is independent of $\aa$ and \cref{thm:convergence_A} implies the convergence in probability of solutions of \eqref{e.080901} to those of \eqref{e.080601}.
\end{remark}

In particular, \cref{thm:convergence_A} implies that $v^{\aa}-V^{\aa}$ converges in law to $0$. Secondly, we prove:

\begin{theorem}\label{thm:convergence_B}
    Suppose $v_0\in H^2$ and $T>0$. Let $\Lc^{\infty,\aa},\Lc^{\infty}$ denote the laws on $C([0,T];H^1)$ of the solutions $V^{\aa}=(\overline{V}^{\aa},\widetilde{V}^{\aa})$ to \eqref{e.080803} and $V=(\overline{V},\widetilde{V})$ to \eqref{e.080602} respectively with  the same initial data $v_0$. Then 
    \begin{align*}
        \Lc^{\infty,\aa}\to \Lc^{\infty}
    \end{align*} 
weakly as $\aa\to\infty$. 
\end{theorem} 

{\cref{thm:convergence_A} and \cref{thm:convergence_B}, together, then immediately imply the main result of this section:}

\begin{theorem}\label{thm:convergence_C}
 Suppose $v_0\in H^2$ and $T>0$. Let $\Lc^\aa, \Lc^\infty$ denote the laws on $C([0,T];H^1)$ of the solutions $v^{\alpha}=(\bv,\tu)$ to \eqref{e.080901} and $V=(\overline{V},\widetilde{V})$ to \eqref{e.080602} corresponding to initial data $v_0$. Then 
    \[
        \Lc^\aa\to\Lc^{\infty},
    \]
weakly as $\aa\to\infty$.
\end{theorem}

{It therefore suffices to prove \cref{thm:convergence_A} and \cref{thm:convergence_B} to establish \cref{thm:convergence_C}.}

\subsection{Proof of Theorem \ref{thm:convergence_A}}
In the proof we will use the following complete orthonormal basis  of $H$ given by 
    \begin{align}\label{e.080902}
        \{\phi_{\gg}e^{ik\cdotp x}\}_{k\in 2\pi\Zb^3\setminus\{(0,0,0)\},\gg=\pm1},
    \end{align}
    where $\phi_{\gg}$ is the eigenvector of $J$ corresponding to the eigenvalue $\gg \mi$. 
    We also denote $\pi_N$ as the projection from $H$ onto the subspace spanned by the basis with magnitude of wave-numbers 
    not exceeding $N$.  For notational convenience, we omit the dependence of $v^{\aa}, V^{\aa}$ on $\aa$ throughout the proof.
\begin{proof}
     Set $\overline{R} =\bv-\overline{V}$ and $\widetilde{R} = \tu-\widetilde{V}$. From equations  \eqref{e.080901} and \eqref{e.080803} we have 
    \begin{align}\label{e.080801}
        \begin{split}
            \partial_t\overline{R} + \Pc_h\left(\overline{R}\cdot\nabla_h  \bv+\overline{V}\cdot\nabla_h  \overline{R}\right)+\overline{I}_{\alpha}(\tu,\tu)=\nu\Delta_h \overline{R}, 
        \end{split}
    \end{align}
    and 
    \begin{align}\label{e.080802}
        \begin{split}
            \partial_t\widetilde{R}+\overline{R}\cdot\nabla_h \tu + \overline{V}\cdot\nabla_h  \widetilde{R} + \frac12\widetilde{R}^{\perp}(\nabla_h ^{\perp}\cdot \bv)+\frac12\widetilde{V}^{\perp}(\nabla_h ^{\perp}\cdot \overline{R})+\widetilde{I}_{\alpha}(\bv,\tu) = \nu\Delta \widetilde{R},
        \end{split}
    \end{align}
    with initial data $(\overline{R}(0),\widetilde{R}(0))=(0,0)$.  Here $\overline{I}_{\alpha}(\tu,\tu)$ and $\widetilde{I}_{\alpha}(\bv,\tu)$ denote the highly oscillatory terms through:
    \begin{align}\label{def:Ialphabar}
        \overline{I}_{\alpha}(\widetilde{f},\widetilde{g}) = \Pc_h \Big(e^{-2\aa Jt}\bP\left(\widetilde{f}\cdot \nabla_h  \widetilde{g} - \widetilde{f}^{\perp}\cdot \nabla_h \widetilde{g}^{\perp}\right) \Big),
    \end{align}
    and  
    \begin{align}\label{e.081103}
        \begin{split}
            \widetilde{I}_{\alpha}(\bv,\tu) &= \frac12e^{2\alpha Jt}\left(\tu\cdot \nabla_h  \bv -\tu\cdot \nabla_h ^{\perp}\bv^{\perp} \right)+\frac12e^{-\aa Jt}(I-2{\bP})\left(\tu\cdot \nabla_h  \tu -\tu^{\perp}\cdot \nabla_h \tu^{\perp}\right)\\
            & + \frac12 e^{\aa Jt}\left(\tu\cdot \nabla_h  \tu -\tu\cdot \nabla_h ^{\perp}\tu^{\perp}+\bP\nabla_h |\tu|^2\right) +w(e^{-\aa Jt}\tu)\partial_{z}\tu = \sum_{i=1}^4\widetilde{I}_{\alpha,i}.
        \end{split}
    \end{align}

    From \eqref{e.080801}, we infer
    \begin{align*}
        \frac12\frac{d}{dt}\|\overline{R}\|_1^2+\nu \|\overline{R}\|_2^2=-\langle \LL \Pc_h\left(\overline{R}\cdot\nabla_h  \bv+\overline{V}\cdot\nabla_h  \overline{R}\right),\LL \overline{R}\rangle - \langle\LL\overline{I}_{\alpha}(\tu,\tu), \LL\overline{R} \rangle.
    \end{align*}
    An application of the Cauchy-Schwarz inequality, interpolation, and Young's inequality implies, for any $\ee>0$
    \begin{align*}
        |\langle \overline{R}\cdot\nabla_h  \bv,\LL^2 \overline{R}\rangle|\leq C\|\overline{R}\|_1\|\bv\|_{1+\ee}\|\overline{R}\|_{2}\leq \frac{C}{\nu}\|\overline{R}\|_1^2\|\bv\|_{1+\ee}^2 + \frac{\nu}{4}\|\overline{R}\|_2^2,  
    \end{align*}
    and  
    \begin{align*}
        |\langle \overline{V}\cdot\nabla_h  \overline{R},\LL^2 \overline{R}\rangle|\leq C\|\overline{V}\|_{1+\ee}\|\overline{R}\|_1\|\overline{R}\|_{2}\leq \frac{C}{\nu}\|\overline{V}\|_{1+\ee}^2\|\overline{R}\|_1^2 + \frac{\nu}{4}\|\overline{R}\|_2^2.
    \end{align*}
    Upon combining these inequalities, we obtain
    \begin{align*}
        |\langle \LL \Pc_h\left(\overline{R}\cdot\nabla_h  \bv+\overline{V}\cdot\nabla_h  \overline{R}\right),\LL \overline{R}\rangle|\leq \frac{C}{\nu}\left(\|\bv\|_{1+\ee}^2+\|\overline{V}\|_{1+\ee}^2\right)\|\overline{R}\|_1^2+\frac{\nu}{2}\|\overline{R}\|_2^2. 
    \end{align*}
    Therefore, since $\overline{R}(0)=0$, we have 
    \begin{align}\label{e.081101}
        \frac12\|\overline{R}(t)\|_1^2+\frac{\nu}{2}  \int_0^t\|\overline{R}(r)\|_2^2dr\leq \frac{C}{\nu}\int_0^t\left(\|\bv\|_{1+\ee}^2+\|\overline{V}\|_{1+\ee}^2\right)\|\overline{R}\|_1^2dr -\int_0^t \langle\LL\overline{I}_{\alpha}(\tu,\tu), \LL\overline{R} \rangle dr. 
    \end{align}

    We now consider the term involving $\overline{I}_{\alpha}(\tu,\tu)$ by taking the approach as in \cite{flandoli2012stochastic}. Then we have the decomposition 
    \begin{align*}
        \langle\overline{I}_{\alpha}(\tu,\tu), \LL^2\overline{R} \rangle &= \langle\LL\overline{I}_{\alpha}(\tu,\tu), \LL\overline{R} \rangle\\
        & = \langle(1-\pi_N)\LL\overline{I}_{\alpha}(\tu,\tu), \LL\overline{R} \rangle + \langle\pi_N\LL\overline{I}_{\alpha}(\tu,(1-\pi_N)\tu), \LL\overline{R} \rangle \\
        &\quad + \langle\pi_N\LL\overline{I}_{\alpha}((1-\pi_N)\tu,\tu), \LL\overline{R} \rangle + \langle\pi_N\LL\overline{I}_{\alpha}(\pi_N\tu,\pi_N\tu), \LL\overline{R} \rangle\\
        &\quad - \langle\pi_N\LL\overline{I}_{\alpha}((1-\pi_N)\tu,(1-\pi_N)\tu), \LL\overline{R} \rangle=\sum_{i=1}^5\overline{I}_{\alpha,i}.
    \end{align*}
   Observe that by the Poincar\'e inequality $\|(1-\pi_N)v\|_{2-\ee}\leq CN^{-\varepsilon}\|v\|_2$. 
   It follows that
   \begin{align*}
    \overline{I}_{\aa,1}\leq C\|\LL^{\ee}\bP\left(\tu\cdot\nabla_h  \tu\right)\|\|(1-\pi_N)\overline{R}\|_{2-\ee}
    \leq CN^{-\varepsilon}\|\LL^{\ee}\left(\tu\cdot\nabla_h  \tu\right)\|\|\overline{R}\|_2, 
   \end{align*}
   On the other hand, by \cref{l.080801} and the Sobolev embedding $H^{3/4}(\mathbb T^3)\subset L^4(\mathbb T^3)$, with small enough $\ee$ we have
   \begin{align*}
    \|\LL^{\ee}\left(\tu\cdot\nabla_h  \tu\right)\|\lesssim \|\LL^{\ee}\tu\|_{L^4}\|\LL^{\ee}\nabla_h  \tu\|_{L^4}\lesssim \|\tu\|_{2}^2. 
   \end{align*}
   Thus, by Young's inequality, we deduce
   \begin{align*}
    \overline{I}_{\aa,1}\leq CN^{-2\varepsilon}\|\tu\|_{2}^4+\frac{\nu}{8}\|\overline{R}\|_2^2.
   \end{align*}
   Next, by interpolation, the Poincar\'e inequality, and Young's inequality we have  
   \begin{align*}
    \overline{I}_{\aa,2}\leq C\|\tu\cdot\nabla_h  (1-\pi_N)\tu\|\|\overline{R}\|_2&\leq C\|\tu\|_{L^{\infty}}\|\nabla_h  (1-\pi_N)\tu\|\|\overline{R}\|_2
    \\
    &\leq CN^{-\varepsilon}\|\tu\|_2\|\tu\|_{1+\ee}\|\overline{R}\|_2\leq CN^{-2\varepsilon}\|\tu\|_2^2\|\tu\|_{1+\ee}^2 + \frac{\nu}{8}\|\overline{R}\|_2^2. 
   \end{align*}
   Similarly, 
   \begin{align*}
    \overline{I}_{\aa,3}\leq C\|((1-\pi_N)\tu)\cdot\nabla_h  \tu\|\|\overline{R}\|_2&\leq C\|(1-\pi_N)\tu\|_{L^{\infty}}\|\tu\|_1\|\overline{R}\|_2\\
    &\leq CN^{-\varepsilon}\|\tu\|_{2}\|\tu\|_1\|\overline{R}\|_2
    \leq CN^{-2\varepsilon}\|\tu\|_2^2\|\tu\|_{1}^2 + \frac{\nu}{8}\|\overline{R}\|_2^2, 
   \end{align*}
   and 
   \begin{align*}
    \overline{I}_{\aa,5}\leq  CN^{-2\varepsilon}\|\tu\|_2^2\|\tu\|_{1}^2 + \frac{\nu}{8}\|\overline{R}\|_2^2. 
   \end{align*}
   To estimate $\overline{I}_{\aa,4}$, we expand it with respect to the complete orthonormal basis  of $H$ given by \eqref{e.080902}. To that end, we let 
   \begin{align*}
    \tu = \sum_{k, k_3\neq 0, \gg}\tu_{k,\gg}\phi_{\gg}e^{ik\cdotp x},\quad v = \sum_{k,\gg}v_{k,\gg}\phi_{\gg}e^{ik\cdotp x}, \quad V = \sum_{k,\gg}V_{k,\gg}\phi_{\gg}e^{ik\cdotp x}.
   \end{align*}
   Observe then that
   \begin{align*}
    \tu^{\perp} =J\tu= \sum_{k, k_3\neq 0, \gg}\gg \mi \tu_{k,\gg}\phi_{\gg}e^{ik\cdotp x},\quad e^{2\aa Jt}\overline{R} = \sum_{k,k_3=0,\gg}e^{2\mi\aa\gg  t}(v_{k,\gg}-V_{k,\gg})\phi_{\gg}e^{ik\cdotp x}. 
   \end{align*}
   Therefore, with the notation $\ell=(\ell',\ell_3)$, for $(j,\ell,k)\in\{j+\ell=k, k_3=0, |j|,|k|,|\ell|\leq N\}$  and $\gg, \gg_1,\gg_2\in \{\pm1\}$, 
   \begin{align*}
    \overline{I}_{\aa,4} &= \langle \pi_N \Lambda\overline{I}_{\alpha}(\pi_N\tu,\pi_N\tu), \LL\overline{R} \rangle = \langle \pi_N\tu\cdot \nabla_h  \pi_N\tu - \pi_N\tu^{\perp}\cdot \nabla_h \pi_N\tu^{\perp}, {\bP}e^{2\aa Jt}\pi_N\LL^2\Pc_h\overline{R} \rangle \\
    & = \langle \pi_N\tu\cdot \nabla_h  \pi_N\tu - \pi_N\tu^{\perp}\cdot \nabla_h \pi_N\tu^{\perp}, e^{2\aa Jt}\pi_N\LL^2\overline{R} \rangle \\
    & = \sum_{\substack{j,\ell,k,\gg_1,\gg_2\\ k_3=0}}\left(\tu_{j,\gg_1}\phi_{\gg_1}\cdot(\mi \ell')\tu_{\ell,\gg_2}-\gg_1\mi\tu_{j,\gg_1}\phi_{\gg_1}\cdot(\mi \ell')\gg_2\mi\tu_{\ell,\gg_2}\right)e^{-2\mi\aa\gg_2 t}(-|k|^2)(v_{-k,\gg_2}-V_{-k,\gg_2})\\
    & = \sum_{\substack{j,\ell,k,\gg_1,\gg_2\\ k_3=0}}\phi_{\gg_1}\cdot(-\mi \ell'|k|^2)\left(1+\gg_1\gg_2\right)e^{-2\mi\aa\gg_2 t}\tu_{j,\gg_1}\tu_{\ell,\gg_2}(v_{-k,\gg_2}-V_{-k,\gg_2})\\
    &=\sum_{\substack{j,\ell,k\\ \gg_\in\{\pm1\}\\ k_3=0}}2\phi_{\gg}\cdot(-\mi \ell'|k|^2)e^{-2\mi\aa\gg t}\tu_{j,\gg}\tu_{\ell,\gg}(v_{-k,\gg}-V_{-k,\gg})\\
    &=\sum_{\substack{j,\ell,k\\ \gg_\in\{\pm1\}\\ k_3=0}}2\phi_{\gg}\cdot(-\mi \ell'|k|^2)e^{-2\mi\aa\gg t}f_{j,\ell,k,\gg},
   \end{align*}
   where 
   \begin{align*}
    f_{j,\ell,k,\gg}&=\tu_{j,\gg}\tu_{\ell,\gg}(v_{-k,\gg}-V_{-k,\gg})\\
    & = \langle \tu, \phi_{\gg}e^{\mi j\cdot x}\rangle \langle \tu, \phi_{\gg}e^{\mi \ell\cdot x}\rangle\left(\langle v, \phi_{\gg}e^{-\mi k\cdot x}\rangle-\langle V, \phi_{\gg}e^{-\mi k\cdot x}\rangle\right).
   \end{align*}
   It follows from Young's inequality and \cref{l.080901}  that 
   \begin{align*}
    \Eb\sup_{t\in[0,T]}|f_{j,\ell,k,\gg}|&\leq C\left(\Eb\sup_{t\in[0,T]}|\tu_{j,\gg}|^2+\Eb\sup_{t\in[0,T]}|\tu_{\ell,\gg}|^4+\Eb\sup_{t\in[0,T]}|v_{-k,\gg}|^4+\Eb\sup_{t\in[0,T]}|V_{-k,\gg}|^4\right)\\
    &\leq C\left(\Eb\sup_{t\in[0,T]}\|\tu\|^2+ \Eb\sup_{t\in[0,T]}\|\tu\|^4+\Eb\sup_{t\in[0,T]}\|v\|^4+\Eb\sup_{t\in[0,T]}\|V\|^4\right)\leq C. 
   \end{align*}
   For any $s,t\in[0,T]$ and function $\varphi$, denote $\DD_{st}\varphi=|\varphi(t)-\varphi(s)|$. Note that  
   \begin{align*}
    &|f_{j,\ell,k,\gg}(t)-f_{j,\ell,k,\gg}(s)|\\
    & \leq  (\DD_{st}\tu_{j,\gg})|\tu_{\ell,\gg}(t)|(|v_{-k,\gg}(t)|+|V_{-k,\gg}(t)|)+|\tu_{j,\gg}(s)|(\DD_{st}\tu_{\ell,\gg})(|v_{-k,\gg}(t)|+|V_{-k,\gg}(t)|)\\
    &\quad + |\tu_{j,\gg}(s)||\tu_{\ell,\gg}(s)|(\DD_{st}v_{-k,\gg}+\DD_{st}V_{-k,\gg}).
   \end{align*}
   By H\"older's inequality and \cref{l.080901}, we have
   \begin{align*}
    &\Eb(\DD_{st}\tu_{j,\gg})|\tu_{\ell,\gg}(t)|(|v_{-k,\gg}(t)|+|V_{-k,\gg}(t)|)\\
    &\leq C\left(\Eb\left(\DD_{st}\tu_{j,\gg}\right)^2\right)^{\frac12}\left(\Eb\left|\tu_{\ell,\gg}\right|^4\right)^{\frac14}\left(\left(\Eb\left|v_{-k,\gg}\right|^4\right)^{\frac14}+\left(\Eb\left|V_{-k,\gg}\right|^4\right)^{\frac14}\right)\\
    &\leq C\left(\Eb\left|\tu_{j,\gg}(t)-\tu_{j,\gg}(s)\right|^2\right)^{\frac12}\left(\Eb\sup_{t\in[0,T]}\left\|\tu\right\|^4\right)^{\frac14}\left(\left(\Eb\sup_{t\in[0,T]}\left\|v\right\|^4\right)^{\frac14}+\left(\Eb\sup_{t\in[0,T]}\left\|V\right\|^4\right)^{\frac14}\right)\\
    &\leq C|t-s|^{\bb}
   \end{align*}
   for some $\bb>0$. Upon applying this argument to the remaining terms we obtain 
   \begin{align*}
    \Eb|f_{j,\ell,k,\gg}(t)-f_{j,\ell,k,\gg}(s)|\leq C|t-s|^{\bb},
   \end{align*}
    where we have applied a similar result of (2) in \cref{{l.080901}} for $V^\alpha$. The proof of the latter is straightforward and we omit it.
   By \cref{l.081101} we have 
   \begin{align}\label{e.081102}
    \Eb\sup_{t\in[0,T]}\left|\int_0^t\overline{I}_{\aa,4}dr\right|&\leq \sum_{j,\ell,k,\gg\in\{\pm 1\}}C_{\ell,k,\gg}\Eb\sup_{t\in[0,T]}\left|\int_0^te^{-2\mi\aa\gg r}f_{j,\ell,k,\gg}(r)dr\right|\notag
    \\
    &\leq C_N\left(\frac{1}{\aa^{\bb}}+\frac{1}{\aa}\right),
   \end{align}
   due to $(j,\ell,k)\in\{j+\ell=k, k_3=0, |j|,|k|,|\ell|\leq N\}$.
   
   Collecting the estimate \eqref{e.081101} and estimates for $\overline{I}_{\aa,i}$ we arrive at 
   \begin{align}
    \begin{split}
        &\frac12\|\overline{R}(t)\|_1^2
        \\
        &\leq C\int_0^t\left(\|\bv(r)\|_{1+\ee}^2+\|\overline{V}(r)\|_{1+\ee}^2\right)\|\overline{R}(r)\|_1^2dr+CN^{-2\ee}\int_0^t\|\tu(r)\|_{2}^4dr + \sup_{t\in[0,T]}\left|\int_0^t\overline{I}_{\aa,4}(r)dr\right|.
    \end{split} 
    \end{align}
    Let $\dd>0$. By \cref{l.080902}, there exist $C_{\dd,1}$ and  $\Om_1\subset \Om$ such that $\Pb(\Om_1)\geq 1-\dd/4$ and 
    \begin{align*}
        \sup_{t\in[0,T]}\|v(t)\|_{1+\ee}^2\leq C_{\dd, 1}, \, \text{ on } \Om_1. 
    \end{align*}
    The same lemma ensures the existence of $C_{\dd,2}$ and $\Om_2\subset \Om$ with $\Pb(\Om_2)\geq 1-\dd/4$ such that
    \begin{align*}
        \sup_{t\in[0,T]}\|\tu(t)\|_{2}^4\leq C_{\dd, 2}, \, \text{ on } \Om_2.
    \end{align*}
    Similarly, by the moment bounds for $V$ in \cref{l.092501}, there are $C_{\dd,3}$ and $\Om_3\subset \Om$ with $\Pb(\Om_3)\geq 1-\dd/4$ such that 
    \begin{align*}
        \sup_{t\in[0,T]}\|V(t)\|_{1+\ee}^2\leq C_{\dd,3}, \, \text{ on } \Om_3. 
    \end{align*}
    Therefore, on $\bigcap_{i=1}^3\Om_i$, we have 
    \begin{align*}
        \|\overline{R}(t)\|_1^2\leq C(\dd)\int_0^t\|\overline{R}(r)\|_1^2dr+ \frac{CTC_{\dd,2}}{N^{2\ee}}+ 2\sup_{t\in[0,T]}\left|\int_0^t\overline{I}_{\aa,4}dr\right|.
    \end{align*}
    Let 
    \[\Om_4=\left\{\om\in\Om: 2\sup_{t\in[0,T]}\left|\int_0^t\overline{I}_{\aa,4}dr\right|\leq \frac{1}{N^{2\ee}}\right\}.\]
    Then by \eqref{e.081102} and Markov's inequality, we have 
    \begin{align*}
        \Pb(\Om_4^c)\leq C_N\left(\frac{1}{\aa^{\bb}}+\frac{1}{\aa}\right). 
    \end{align*}
    Hence on $\bigcap_{i=1}^4\Om_i$ we have 
    \begin{align*}
        \|\overline{R}(t)\|_1^2\leq C(\dd)\int_0^t\|\overline{R}(r)\|_1^2dr+ \left(CTC_{\dd,2}+1\right)\frac{1}{N^{2\ee}}.
    \end{align*}
    By Gr\"onwall's inequality, it then follows that
    \begin{align*}
        \sup_{t\in[0,T]} \|\overline{R}(t)\|_1^2\leq \left(CTC_{\dd,2}+1\right)\frac{e^{TC(\dd)}}{N^{2\ee}}
    \end{align*}
    
    Now choose $N_0(T)$ sufficiently large such that for $N\geq N_0(T)$, 
    \[
        \left(CTC_{\dd,2}+1\right)\frac{e^{TC(\dd)}}{N^{2\ee}}\leq \dd.
    \] 
    Given $N\geq N_0(T)$, choose $\aa_0(N)$ such that for all $\aa\geq \aa_0(N)$,
    \[\Pb(\Om_4^c)\leq C_N\left(\frac{1}{\aa^{\bb}}+\frac{1}{\aa}\right)<\frac{\dd}{4}.\] 
    Then we have 
    \begin{align*}
        \sup_{t\in[0,T]} \|\overline{R}(t)\|_1^2\leq \dd
    \end{align*}
    on $\bigcap_{i=1}^4\Om_i$ with probability
    \begin{align*}
        \Pb\left(\bigcap_{i=1}^4\Om_i\right)\geq 1-\sum_{i=1}^4\Pb\left(\Om_i^c\right)\geq 1-\dd.
    \end{align*}
    In particular,
    \begin{align}\label{e.092101}
        \Pb\left(\sup_{t\in[0,T]} \|\overline{R}(t)\|_1^2\leq \dd\right)\geq 1- \dd, 
    \end{align}
    for all $\aa\geq \aa_0$. 

    Similarly, we can prove that $\sup_{t\in[0,T]} \|\widetilde{R}(t)\|_1^2$ converges to $0$ in probability as $\aa\to\infty$. Indeed, from \eqref{e.080802} we have 
    \begin{align*}
        &\frac12\frac{d}{dt}\|\widetilde{R}\|_1^2+\nu\|\widetilde{R}\|_2^2\notag
        \\
        &=-\left\langle \LL\left(\overline{R}\cdot\nabla_h \tu + \overline{V}\cdot\nabla_h  \widetilde{R} + \frac12\widetilde{R}^{\perp}(\nabla_h ^{\perp}\cdot \bv)+\frac12\widetilde{V}^{\perp}(\nabla_h ^{\perp}\cdot \overline{R})+\widetilde{I}_{\alpha}(\bv,\tu)\right),\LL\widetilde{R}\right\rangle 
    \end{align*}
    with $\widetilde{I}_{\alpha}(\bv,\tu)$ given by \eqref{e.081103}. We will focus on treating the worst term $\widetilde{I}_{\alpha,4}$; the remaining terms, $\widetilde{I}_{\alpha,1}$, $\widetilde{I}_{\alpha,2}$, $\widetilde{I}_{\alpha,3}$ would then be readily handled. We infer from \cref{l.081002} that  
    \begin{align*}
        \widetilde{I}_{\alpha,4} =\frac12 e^{-\aa Jt}\left(w(\tu) \pp_z\tu-w(\tu^{\perp}) \pp_z\tu^{\perp}\right) + \frac12 e^{\aa Jt}\left(w(\tu) \pp_z\tu+w(\tu^{\perp}) \pp_z\tu^{\perp}\right).
    \end{align*}
    Note that we only need to deal with one of the four terms above since they share a common structure. Let 
    \[I_{\alpha}(\widetilde{f},\widetilde{g}) =  e^{\aa Jt}w(\widetilde{f}) \pp_z\widetilde{g}.\]
    Then 
    
    \begin{align*}
        \langle\LL I_{\alpha}(\tu,\tu),\LL\widetilde{R}\rangle
        & =\langle(1-\pi_N)\LL I_{\alpha}(\tu,\tu), \LL\widetilde{R} \rangle + \langle\pi_N\LL I_{\alpha}(\tu,(1-\pi_N)\tu), \LL\widetilde{R} \rangle \\
        &\quad + \langle\pi_N\LL I_{\alpha}((1-\pi_N)\tu,\tu), \LL\widetilde{R} \rangle + \langle\pi_N\LL I_{\alpha}(\pi_N\tu,\pi_N\tu), \LL\widetilde{R} \rangle\\
        &\quad - \langle\pi_N\LL I_{\alpha}((1-\pi_N)\tu,(1-\pi_N)\tu), \LL\widetilde{R} \rangle = :\sum_{i=1}^5I_{\alpha,i}. 
    \end{align*}
    Again by the Poincar\'e inequality, \cref{l.080801} and Sobolev embedding $H^{3/4}(\mathbb T^3)\subset L^4(\mathbb T^3)$, we have 
    \begin{align*}
        I_{\alpha,1}&\leq C\|\LL^{\ee}\left(w(\tu) \pp_z\tu\right)\|\|\LL^{2-\ee}(1-\pi_N)e^{-\aa Jt}\widetilde{R}\|\\
        &\leq \frac{C}{N^{\ee}}\|\widetilde{R}\|_{2}\left(\|\LL^{\ee}w(\tu)\|_{L^{4}}\|\pp_z\tu\|_{L^4}+\|w(\tu)\|_{L^4}\|\LL^{\ee}\pp_z\tu\|_{L^{4}}\right)\\
        &\leq \frac{C}{N^{\ee}}\|\widetilde{R}\|_{2}\|\tu\|_{2}^2. 
    \end{align*}
    By the standard estimate $\|uv\|\leq C \|u\|_{s}\|v\|_{3/2-s}$ for $s\in(0,3/2)$, one has 
    \begin{align*}
        I_{\alpha,2}&\leq C\|w(\tu) \pp_z(1-\pi_N)\tu\|\|\LL^2e^{-\aa Jt}\widetilde{R}\|\\
        &\leq C\|\widetilde{R}\|_{2}\|w(\tu)\|_{1}\|\pp_z(1-P_N)\tu\|_{\frac12}\\
        &\leq \frac{C}{N^{\ee}}\|\widetilde{R}\|_{2}\|\tu\|_{2}\|\tu\|_{3/2+\ee},
    \end{align*}
    and similarly both $I_{\alpha,3}, I_{\alpha,5}$ can be bounded as $I_{\alpha,2}$. The term $I_{\alpha,4}$ can be bounded  similar to \eqref{e.081102}. Therefore, $\sup_{t\in[0,T]} \|\widetilde{R}(t)\|_1^2$ converges in probability to $0$ by arguing in the same way as we did in establishing \eqref{e.092101}. 

\end{proof}

\subsection{Proof of Theorem \ref{thm:convergence_B}}\label{s.100701}

    It is standard that for $v_0\in H^2$, both equations \eqref{e.080803} and \eqref{e.080602} have global solutions such that, almost surely, the solution trajectories are in $C([0,T];H^2)\cap L^{2}(0,T;H^3)$. We denote by
    \[
        M_{\aa}(t) =\int_0^t G_{\aa}(r)dW(r).
    \]
    the Gaussian martingale, where
    \[
        G_{\aa}(r)=\Pc_h{\bP}\ss+e^{\aa Jr}(I-\bP)\ss.
    \]
    As a preliminary lemma, we identify the covariance operator of $M_{\aa}$.
    
    \begin{lemma}\label{lem:covariance}
    Suppose that $\ss$ commutes with $\bP$. Then, given $f,g\in H$ and $s,t\geq 0$, one has
        \begin{align*}
        &\Eb[\langle M_\aa(t),f\rangle\langle M_\aa(s),g\rangle]\notag
        \\
        &=(t\wedge s)\left\langle \ss \bP\Pc_h f, \ss \bP\Pc_h g\right\rangle + \int_0^{t\wedge s}\left\langle \ss(I-\bP)e^{-\aa Jr}f, \ss(I-\bP)e^{-\aa Jr}g\right\rangle.
        \end{align*}
    \end{lemma}
    
    \begin{proof}
    Observe that 
    \begin{align*}
        &\Eb[\langle M_\aa(t),f\rangle\langle M_\aa(s),g\rangle]\\
        & = \Eb \left[\left\langle \int_0^t G_{\aa}(r)dW(r),f\right\rangle \left\langle \int_0^s G_{\aa}(r)dW(r),g\right\rangle\right]\\
        & = \Eb \left[\int_0^t\sum_k \left\langle G_{\aa}(r)\boldsymbol{e}_k,f\right\rangle dW_k(r) \int_0^s \sum_k\left\langle G_{\aa}(r)\boldsymbol{e}_k,g\right\rangle dW_k(r)\right]\\
        & = \int_0^{t\wedge s}\sum_{k}\left\langle G_{\aa}(r)\boldsymbol{e}_k,f\right\rangle \left\langle G_{\aa}(r)\boldsymbol{e}_k,g\right\rangle dr \\
        & = \int_0^{t\wedge s}\sum_{k}\left\langle (\Pc_h{\bP}\ss+e^{\aa Jr}(I-\bP)\ss)\boldsymbol{e}_k,f\right\rangle \left\langle (\Pc_h{\bP}\ss+e^{\aa Jr}(I-\bP)\ss)\boldsymbol{e}_k,g\right\rangle dr \\
        & =(t\wedge s)\left\langle \ss \bP\Pc_h f, \ss \bP\Pc_h g\right\rangle + \int_0^{t\wedge s}\left\langle \ss(I-\bP)e^{-\aa Jr}f, \ss(I-\bP)e^{-\aa Jr}g\right\rangle, 
    \end{align*}
    where we used the following orthogonality at the last step
    \begin{align*}
        \left\langle \ss(I-\bP)e^{-\aa Jr}f, \ss \bP \Pc_hg \right\rangle =0, \quad f, g\in H, 
    \end{align*}
    which holds owing to the assumption that $\ss$ commutes with $\bP$. 
    \end{proof}

We will also make use of \cite[Theorem A2]{flandoli2012stochastic}, which allows one to identify Cesaro averages associated to strongly continuous unitary groups on Hilbert spaces as covariance operators. For the sake of completeness, we restate this result here.

\begin{lemma}[\hspace{-0.01pt}\cite{flandoli2012stochastic}]\label{lem:Cesaro}
Let $\Upsilon=(\Upsilon)_{t\in\mathbb{R}}$ be a strongly continuous unitary group on a Hilbert space $H$. 
If $Q$ is a positive trace class operator, then
    \begin{align}\notag
        C_Q:=\lim_{T\to\infty}\frac{1}T\int_0^T\Upsilon_tQ\Upsilon_{-t}dt,
    \end{align}
exists in operator norm, where $C_Q$ is a positive trace class operator.
\end{lemma}

We are now ready to prove \cref{thm:convergence_B}.

\begin{proof}[Proof of Theorem \ref{thm:convergence_B}]
    The proof is divided into three steps. First, we establish tightness of the stochastic integrals $\{M_{\aa}\}_{\aa}$ and identify its limit. Second, we establish tightness of the family of laws $\{\Lc^{\infty,\aa}\}_{\aa}$ corresponding to the auxiliary system \eqref{e.080803}. Lastly, we prove convergence of $\Lc^{\infty,\aa}$ to $\Lc^{\infty}$.
    
    \subsubsection*{Step 1: Tightness and the limit of the stochastic integral $M_{\aa}$} By the Burkholder-Davis-Gundy inequality, we have for $p>1$, 
    \begin{align*}
        \Eb\|M_{\aa}(t)-M_{\aa}(s)\|_1^p\leq C\left(\int_s^t\|G_{\aa}(r)\|_{L_2(H,H^1)}^2 dr\right)^{p/2}\leq C\|\ss\|_{L_2(H,H^1)}^p|t-s|^{p/2}. 
    \end{align*}
    Hence,  for any $\gg<\frac12$  and $p>1$, 
    \begin{align*}
        \Eb\left[\int_0^T\int_0^T\frac{\|M_{\aa}(t)-M_{\aa}(s)\|_1^p}{|t-s|^{1+\gg p}}dtds\right]\leq \int_0^T\int_0^T\frac{C}{|t-s|^{1+(\gg-1/2) p}}dtds<\infty
    \end{align*}
    independent of $\aa$. In particular, the family $\{M_{\aa}\}_{\aa}$ is uniformly bounded in $L^{1}(\Om; W^{\gg,p}(0,T;H^1))$.  Assuming further that $\gg p>1$, we invoke the fact that $W^{\gg,p}(0,T;H^1)$ is compactly embedded into $C([0,T],H)$. We therefore deduce that $\{M_{\aa}\}_\aa$ is tight in $C([0,T],H)$.

    Next, observe that
    \begin{align*}
        \frac{1}{t\wedge s}\int_0^{t\wedge s}e^{\aa Jr}(I-\bP)\ss^2(I-\bP)e^{-\aa Jr}dr 
        = \frac{1}{(t\wedge s)\aa}\int_0^{(t\wedge s)\aa}e^{J\tau}(I-\bP)\ss^2(I-\bP)e^{-J\tau}d\tau.
    \end{align*}
    We then invoke \cref{lem:Cesaro}, with $\Upsilon_t=e^{Jt}(I-\bP)$, to deduce that as $\alpha\to\infty$ this operator converges in operator norm to a positive trace class operator given by 
    \[
        \widetilde{Q} := \lim_{T\to\infty}\frac{1}{T}\int_0^{T}e^{J\tau}(I-\bP)\ss^2(I-\bP)e^{-J\tau}d\tau.
    \]
    We remark that the limit $\widetilde{Q}$ will be the same when $\alpha\to -\infty$.
    Therefore, from \cref{lem:covariance} we deduce 
    \[
        \lim_{\aa\to\infty}\Eb[\langle M_\aa(t),f\rangle\langle M_\aa(s),g\rangle] = (t\wedge s)\left(\left\langle \Pc_h \bP\ss^2 \bP\Pc_h f, g\right\rangle + \langle \widetilde{Q}f, g\rangle \right).
    \]

    Now, since $\{M_{\aa}\}_{\aa}$ is tight in $C([0,T],H)$, from any subsequence one can extract a further subsequence that converges in law to a Gaussian process, $M$, on $C([0,T],H)$ with covariance defined by
    \begin{align*}
        G:=\Pc_h \bP\ss^2 \bP\Pc_h+\widetilde{Q}. 
    \end{align*}
    The covariance uniquely determines $M$, independent of the choice of subsequence. Therefore, the whole sequence $\{M_{\aa}\}_\aa$ converges in law to $M$ in $C([0,T],H)$. Observe that $M$ can subsequently be identified as a Brownian motion $\sqrt{G}\widetilde{W}$ in $H$ with covariance $G$, which is the driving noise of system \eqref{e.080602}  where $\widetilde{W}$ is an independent copy of the cylindrical Wiener process $W$.

    \subsubsection*{Step 2: Tightness of $\{\Lc^{\infty,\aa}\}$}
   First, we show that $\overline{V}^{\aa}\in W^{\gg,p}(0,T;H^1)$, for any $\gg<\frac12$ and $p\geq 2$. Observe that from \eqref{e.080803}, we have
    \begin{align*}
        \overline{V}^{\aa}(t) = \overline{V}_0^{\aa} - \int_0^t\Pc_h(\overline{V}^{\aa}\cdot \nabla_h  \overline{V}^{\aa})dr + \nu\int_0^t\DD \overline{V}^{\aa} dr + \int_0^t \Pc_h\overline{\ss}dW(r).
    \end{align*}
By H\"older's inequality and interpolation, we have
    \begin{align*}
        \|\overline{V}^{\aa}\cdot \nabla_h  \overline{V}^{\aa}\|_1\lesssim \|\overline{V}^{\aa}\|_{L^{\infty}}\|\nabla_h \overline{V}^{\aa}\|_1+\|\overline{V}^{\aa}\|_{1}\|\nabla_h \overline{V}^{\aa}\|_{L^{\infty}}\lesssim \|\overline{V}^{\aa}\|_{2}\|\overline{V}^{\aa}\|_{3}.
    \end{align*}
    It then follows that
    \begin{align*}
        \|\overline{V}^{\aa}(t)-\overline{V}^{\aa}(s)\|_1&\leq \int_s^t\|\overline{V}^{\aa}\cdot \nabla_h  \overline{V}^{\aa}\|_1dr + \nu \int_s^t\|\overline{V}^{\aa}\|_3dr + \left\|\int_s^t \Pc_h\overline{\ss}dW(r)\right\|_1\\
        &\leq C|t-s|^{\frac12}\left(\sup_{t\in[0,T]}\|\overline{V}^{\aa}\|_2+1\right)\left(\int_s^t\|\overline{V}^{\aa}\|_3^2dr\right)^{1/2}+ \left\|\int_s^t \Pc_h\overline{\ss}dW(r)\right\|_1.
    \end{align*}
    By the Burkholder-Davis-Gundy inequality and \cref{l.092501}, we then have
    \begin{align*}
        \Eb\|\overline{V}^{\aa}(t)-\overline{V}^{\aa}(s)\|_1^p&\leq C|t-s|^{p/2}+\Eb\left\|\int_s^t \Pc_h\overline{\ss}dW(r)\right\|_1^p\\
        &\leq C|t-s|^{p/2}+C\|\ss\|_{L_2(H,H^1)}^p|t-s|^{p/2},
    \end{align*}
    for any $p>1$. Thus, given $\gg<\frac12$ and $p\geq 2$, we have
    \begin{align*}
        \Eb\int_0^T\int_0^T\frac{\|\overline{V}^{\aa}(t)-\overline{V}^{\aa}(s)\|_1^p}{|t-s|^{1+\gg p}}dtds\leq C\int_0^T\int_0^T\frac{1}{|t-s|^{1+(\gg-1/2)p}}dtds <\infty,
    \end{align*}
    as claimed.
    
    Next, we show that $\tV^{\aa}\in W^{\gg,p}(0,T;H^1)$ for $\gg<1/2$ and $p\geq 2$ as well.  Indeed, from the equation of $\tV^{\aa}$ we have 
    \begin{align*}
        \begin{split}
            \tV^\alpha(t) = \tV^\alpha(0)-\int_0^t\left(\bV^\alpha\cdot\nabla_h \tV^\alpha+\frac12(\tV^{\alpha})^{\perp}\left(\nabla_h ^{\perp}\cdot\bV^\alpha\right)\right)dr + \nu\int_0^t\Delta \tV^\alpha dr+ \int_0^t e^{\aa Jr}\widetilde{\sigma}dW(r).
        \end{split}
    \end{align*}
    This implies
    \begin{align*}
        \|\tV^\alpha(t)-\tV^\alpha(s)\|_1&\leq \int_s^t\left\|\bV^\alpha\cdot\nabla_h \tV^\alpha+\frac12(\tV^{\alpha})^{\perp}\left(\nabla_h ^{\perp}\cdot\bV^\alpha\right)\right\|_1dr\\
        &+\nu\int_s^t\|\tV\|_3dr +\left\|\int_s^t e^{\aa Jr}\widetilde{\sigma}dW(r)\right\|_1. 
    \end{align*}
    By H\"older's inequality, Ladyzhenskaya's inequality, and the embedding $H^2\subset W^{1,4}$, we see that
    \begin{align*}
    \|\bV^{\aa}\cdot\nabla_h \tV^{\aa}\|_1\lesssim \|\bV^{\aa}\|_{L^{\infty}}\|\|\nabla_h \tV^{\aa}\|_{1}+\|\bV^{\aa}\|_{W^{1,4}}\|\nabla_h \tV^{\aa}\|_{L^4}\lesssim \|\bV^{\aa}\|_2\|\tV^{\aa}\|_{2}.
    \end{align*}
    Similarly, we have
    \begin{align*}
        \left\|(\tV^{\alpha})^{\perp}\left(\nabla_h ^{\perp}\cdot\bV^\alpha\right)\right\|_1\lesssim \|\bV^{\aa}\|_2\|\tV^{\aa}\|_{2}. 
    \end{align*}
    Therefore, by the Burkholder-Davis-Gundy inequality and \cref{l.092501}, we deduce
    \begin{align*}
        \Eb\|\tV^\alpha(t)-\tV^\alpha(s)\|_1^p\leq C|t-s|^{p/2},
    \end{align*}
    or any $\gg<1/2$ and $p\geq 2$, as claimed.

    Observe that for $\gg p>1$, $W^{\gg,p}(0,T;H^1)$ is compactly embedded into $C([0,T],H)$. In addition, the above estimates show that $V^{\aa} = (\bV^{\aa},\tV^{\aa})$ is bounded in $L^1(\Om;W^{\gg,p}(0,T;H^1))$ uniformly in $\aa$, provided that $\gg<1/2$, $p\geq 2$, and $\gg p>1$. Hence the laws, $\Lc^{\infty,\aa}$ corresponding to $V^{\aa}$, form a tight family in $C([0,T],H)$. 
    
    \subsubsection*{Step 3: Convergence of $\Lc^{\infty,\aa}$ to $\Lc^{\infty}$} From the previous two steps, we know that the law of the family $\{(V^{\aa},M_{\aa})\}_{\aa}$ is tight in $C([0,T];H)\times C([0,T];H)$. Thus, from any subsequence, there is a further subsequence whose law converges to some probability measure $\mathcal{M}$. \cref{l.092501} ensures that $\Lc^{\infty,\aa}$ is concentrated on $C([0,T],H^2)$. Thus, by exactly the same argument as in the proof of \cite[Theorem 4]{flandoli2012stochastic}, we infer that any such subsequence of $\Lc^{\infty,\aa}$ converges in $C([0,T],H^1)$. Now recall that in the first step we have shown that $M_{\aa}$ converges in law to $M$. On the other hand, by a standard argument using Skorokhod's representation theorem, one can verify that the velocity component of the limit measure, $\mathcal{M}$, coincides with the law of some solution to equation \eqref{e.080602}. By uniqueness of solutions to \eqref{e.080602}, this law is exactly $\Lc^{\infty}$, which is the law of $V$. Consequently, any convergent subsequence of the family $\{(V^{\aa},M_{\aa})\}_{\aa}$ has $\mathcal{M}$ as its limit. Hence, $\Lc^{\infty,\aa}$ converges weakly to $\Lc^{\infty}$ in $C([0,T],H^1)$. This completes the proof.

\end{proof}

\section{Exponential mixing of the limit resonant system}\label{s.121504}
In this section, we prove that the limit resonant system
\begin{align}\label{e.081201}
    \begin{split}
        &\partial_t\bV+\Pc_h(\bV\cdot\nabla_h  \bV) = \nu\DD_h \bV + {\bP}G \partial_tW,\\
        &\partial_t\tV +\bV\cdot\nabla_h \tV+\frac12\tV^{\perp}\left(\nabla_h ^{\perp}\cdot\bV\right)=\nu\Delta \tV+ (I-\bP)G\partial_tW, 
    \end{split}
\end{align}
is uniquely ergodic in $H^1$, with an exponentially mixing invariant measure $\mu$. Note that for notational convenience, we use $W$ over $(\Om,\Fc,\Pb)$ instead of $\widetilde{W}$, and we denote the noise coefficient as $G$ so that  
\[{\bP}G = \Pc_h\overline{\ss}, \quad (I-\bP)G = \sqrt{\widetilde{Q}}.\] 
For $ k=(k_1,k_2,k_3)\in K:=2\pi\mathbb Z^3\setminus\{(0,0,0)\}$, let $\{(\ll_k, e_k)\}_{k\in K}$ represent the real eigenvalue-eigenvector pairs of $-\Delta$ on $\Tb^3$ as introduced in \eqref{e.011501}. 
The main result of this section is the following theorem.

\begin{theorem}[Exponential mixing in $H^1$]\label{t.121701}
    There exists $N=N(\nu,\sigma)>0$ such that if
    \begin{align*}
        \bP_N H\subset \mathrm{Range}({\bP}G)\quad \text{ and }\quad \tP_N H\subset \mathrm{Range}((I-\bP)G),
    \end{align*}
    then the system \eqref{e.081201} has a unique ergodic invariant measure $\mu\in \Pscr(H^1)$.  Moreover, there are constants $C,c>0$ such that 
    \begin{align*}
        \left|P_{t}\varphi(v_0)-\int_{H^1}\varphi(v)\mu(dv)\right|\leq C(1+\|v_0\|^{2}+\|\bv_0\|_1^{2})\|\varphi\|_{C^{1}(H^1)}e^{-ct},
    \end{align*}
    for all $t\geq0$, $v_0\in H^1$, and $\varphi\in C^{1}(H^1)$.
\end{theorem}

\subsection{Proof of Theorem \ref{t.121701}}

The proof of \cref{t.121701} will proceed in two steps. First we will establish the existence of spectral gap with respect to the Wasserstein distance $\rho$ (\cref{t.081201} below); this implies exponential mixing in $H$ (\cref{c:081201} below). Subsequently, we will bootstrap this result to exponential mixing in $H^1$ by exploiting higher-order moment bounds.

To state the theorem constituting the first step, let us first introduce the the Wasserstein-1 metric on $\Pscr(H^1)$: let $d$ denote the distance 
    \begin{align}\label{def:d}
        d(u,v):=\|u-v\|\wedge 1.
    \end{align}
We denote the Wasserstein-1 metric by $\rho$ given by
    \begin{align}\label{e.121604}
        \rho(\mu_1,\mu_2): = \inf_{\Pi\in \mathcal{C}(\mu_1,\mu_2)}\int_{H^1\times H^1}d(u,v)\Pi(dudv),
    \end{align}
where $\mathcal{C}(\mu_1,\mu_2)$ is the set of couplings between $\mu_1$ and $\mu_2$. We will then prove the following result.

\begin{theorem}[Exponential mixing in $H$]\label{t.081201}
    There exists $N=N(\nu,\sigma)>0$ such that if
    \begin{align}\label{e.121602}
        \bP_N H\subset \mathrm{Range}({\bP}G)\quad \text{ and }\quad \tP_N H\subset \mathrm{Range}((I-\bP)G),
    \end{align}
    then the system \eqref{e.081201} has a unique ergodic invariant measure $\mu\in \Pscr(H^1)$. Moreover, there exist constants $C,c>0$ such that 
    \begin{align*}
        \rho(P_t(v,\cdot),\mu)\leq C(1+\|v\|^2+\|\bv\|_1^2)e^{-c t},
    \end{align*}
    for all $t\geq0, v\in H^1$. 
\end{theorem}

To see that this implies exponential mixing in $H$, let us first introduce the set of bounded H\"older continuous functions on $H^k$, denoted by $C^{\gg}(H^k)$: given $\gg\in (0,1]$, $C^{\gg}(H^k)$ consists of all functions $\phi:H^k\to\mathbb{R}$ such that 
\begin{align*}
    \|\phi\|_{C^{\gg}(H^k)}: =\sup_{u\in H^k}|\phi(u)|+\sup_{\substack{u,v\in H^k\\ 0<\|u-v\|_k\leq 1}}\frac{|\phi(u)-\phi(v)|}{\|u-v\|^\gamma_k}<\infty.
\end{align*}
Note that $C^{\gg_1}(H^k)\subset C^{\gg_2}(H^k)$ if $\gg_1\geq \gg_2$, and when $\gg=1$, $C^{1}$ is the space of bounded Lipschitz continuous functions. Next, recall the distance function $d$ defined in \eqref{def:d}. Then for any $\gg\in (0,1]$, the function $d_{\gg}(u,v):=d(u,v)^{\gg}$ is again a distance. For $k=0$, by $[\phi]_{\gg}$ we denote the H\"older seminorm defined by 
    \[
        [\phi]_{\gg} = \sup_{u\neq v}\frac{|\phi(u)-\phi(v)|}{d_{\gg}(u,v)}.
    \]
One then obtains the following corollary of \cref{t.081201}:
\begin{corollary}\label{c:081201}
    For any $\gg\in (0,1]$, $\phi\in C^{\gg}(H)$ and $v_0\in H^1$, we have 
    \begin{align*}
        \left|P_{t}\phi(v_0)-\int_{H}\phi(v)\mu(dv)\right|\leq C\|\phi\|_{C^{\gg}(H)}(1+\|v_0\|^{2\gg}+\|\bv_0\|_1^{2\gg})e^{-c\gg t}.
    \end{align*}
\end{corollary}
\begin{proof}
Observe that $d_\gg$ induces a Wasserstein metric 
    \begin{align*}
        \rho_{\gg}(\mu_1,\mu_2):&=\inf_{\Pi\in \mathcal{C}(\mu_1,\mu_2)}\int_{H^1\times H^1}d_{\gg}(u,v)\Pi(dudv)\\
        &\leq \left(\inf_{\Pi\in \mathcal{C}(\mu_1,\mu_2)}\int_{H^1\times H^1}d(u,v)\Pi(dudv)\right)^{\gg} = \rho(\mu_1,\mu_2)^{\gg}
    \end{align*}
    for any $\mu_1,\mu_2\in \Pscr(H^1)$ by Jensen's inequality. Therefore, by \cref{t.081201}, we have 
    \begin{align*}
        \rho_{\gg}(P_t(v_0,\cdot),\mu)\leq \rho(P_t(v_0,\cdot),\mu)^{\gg}\leq C(1+\|v_0\|^{2\gg}+\|\bv_0\|_1^{2\gg})e^{-c\gg t}.
    \end{align*}
    On the other hand, by the Monge-Kantorovich
    duality \cite{villani2009optimal}, we have 
    \begin{align*}
        \left|P_{t}\phi(v_0)-\int_{H}\phi(v)\mu(dv)\right|&= \left|\int_{H}\phi(v) P_t(v_0,dv)-\int_{H}\phi(v)\mu(dv)\right|\\
        &\leq [\phi]_{\gg}\sup_{[\varphi]_{\gg}\leq 1}\left|\int_{H}\varphi(v) P_t(v_0,dv)-\int_{H}\varphi(v)\mu(dv)\right|\\
        &=[\phi]_{\gg}\rho_{\gg}(P_t(v_0,\cdot),\mu)\\
        &\leq C\|\phi\|_{C^{\gg}(H)}(1+\|v_0\|^{2\gg}+\|\bv_0\|_1^{2\gg})e^{-c\gg t}.
    \end{align*}
\end{proof}

In order to prove \cref{t.081201}, we will apply the abstract result established in \cite[Theorem 4.2]{butkovsky2020generalized}. For the sake of completeness, we recall this result now. 

\begin{theorem}[\hspace{-0.01pt}\cite{butkovsky2020generalized}]\label{thm:butkovsky}
Let $(E,\rho)$ be a Polish space. For $x\in E$, let $\mathrm{P}_x$ denote the law of the Markov process $X=\{X(t)\}_{t\geq0}$ with $X(0)=x$ and corresponding Markov kernel $\{P_t\}_{t\geq0}$. Suppose that $\{P_t\}_{t\geq0}$ is Feller and assume that there exists a lower semicontinuous function $U:E\to[0,\infty)$, a measurable function $S:E\to[0,\infty]$, and a premetric $q$ on $E$ such that for any given $x,y\in E$, there exists a couple of progressively measurable random processes $X^{x,y}=\{X^{x,y}(t)\}_{t\geq0}$, $Y^{x,y}=\{Y^{x,y}(t)\}_{t\geq0}$ such that the Markov kernel satisfies the following conditions:
    \begin{enumerate}
        \item[(H1)] There exists $\zeta>0$ and $\kappa\geq0$ such that
            \begin{align}\notag
                q(X^{x,y}(t),Y^{x,y}(t))\leq q(x,y)\exp\left(-\zeta t+\kappa\int_0^tS(X^{x,y}(s))ds\right),\quad 
            \end{align}
        for all $t\geq0$.
        \item[(H2)] There exist $\mu>0$ and $b\geq0$ satisfying
            \begin{align}\notag
                \zeta>\kappa\frac{b}{\mu},
            \end{align}
        such that
            \begin{align}\notag
                U(X^{x,y}(t))+\mu\int_0^tS(X^{x,y}(s))ds\leq U(X^{x,y}(0))+bt+M(t),
            \end{align}
        for all $t\geq0$, where $M$ is a continuous local martingale with $M(0)=0$ and quadratic variation satisfying
            \begin{align}\notag
                d[M](t)\leq b_1S(X^{x,y}(t))dt+b_2dt,
            \end{align}
        for some $b_1,b_2\geq0$, for all $t\geq0$.
        \item[(H3)] $\mathrm{P}_x=\mathrm{Law}(X^{x,y})$ and for every $\delta\in(0,1]$, there exists a constant $C_\delta>0$ such that
            \begin{align}\notag
                d_{TV}\left(\mathrm{Law}(Y^{x,y}(t),P_t(y,\cdot)\right)\leq C_\delta\left(\mathbb{E}\left(\int_0^tq(X^{x,y}(s),Y^{x,y}(s))ds\right)^\delta\right)^{1/2},
            \end{align}
        for all $t\geq0$. Additionally, for any $\delta\in(0,1]$, $R>0$, there exists $\varepsilon=\varepsilon(R,\delta)>0$ such that 
            \begin{align}\notag
                \mathbb{E}\left(\int_0^tq(X^{x,y}(s),Y^{x,y}(s))ds\right)^\delta\quad\text{implies}\quad d_{TV}\left(\mathrm{Law}(Y^{x,y}(t),P_t(y,\cdot)\right)\leq 1-\varepsilon.
            \end{align}
    \end{enumerate}
    If there exists a measurable function $F:E\to[0,\infty)$ and constants $\gamma, K>0$ such that
        \begin{align}\notag
            \mathbb{E}F(X(t))\leq V(x)-\gamma \mathbb{E}\int_0^tF(X(s))ds+Kt,
        \end{align}
    for all $t\geq0$, $x\in E$, and $U(\cdot)$ and $q(\cdot,\cdot)$ are bounded on the level sets $\{F\leq M\}$ and $\{F\leq M\}\times\{F\leq M\}$, respectively, and $\rho\leq q^\delta$, for some $\delta>0$, then $\{P_t\}_{t\geq0}$ has a unique invariant measure, $\mu$. Moreover, there exist constants $C,r>0$ such that
        \begin{align}\notag
            W_{\rho\wedge1}\left(P_t(x,\cdot),\mu\right)\leq C(1+F(x))e^{-rt},
        \end{align}
    for all $t\geq0$, $x\in E$.
\end{theorem}

Let us now prove \cref{t.081201}.

\begin{proof}[Proof of \cref{t.081201}]
    Let $V=(\bV,\tV)$ be the solution to \eqref{e.081201} with initial condition $V_0\in H^1$. Let  $\Vc=(\bVc,\tVc)$ denote the solution to the following controlled stochastic system with initial condition $\Vc_0\in H^1$:
    \begin{align}\label{e.121501}
        \begin{split}
            &\partial_t\bVc+\Pc_h(\bVc\cdot\nabla_h  \bVc) = \nu\DD_h \bVc - \frac{\nu \ll_N}{2}\bP_N\left(\bV-\bVc \right)+ {\bP}G\partial_tW,\\
            &\partial_t\tVc +\bVc\cdot\nabla_h \tVc+\frac12\tVc^{\perp}\left(\nabla_h ^{\perp}\cdot\bVc\right)=\nu\Delta \tVc- \frac{\nu \ll_N}{2}\tP_N\left(\tV-\tVc \right) + (I-\bP)G\partial_tW.
        \end{split} 
    \end{align}
    We proceed to verify the assumptions of \cref{thm:butkovsky}. Let $(E,\rho)=(H,d)$. Also, let 
        \begin{align}\label{def:USq}
            U(v) = \|v\|^2+\|\bv\|_1^2, \quad
            S(v) = 1+\|v\|_1^2+\|\bv\|_{2}^2,\quad q(u,v)=\|u-v\|^2+\|\bu-\bv\|_1^2.
        \end{align}

\subsubsection*{Verifying Assumption $(\mathrm{H1})$.} We claim that there exists $C>0$ such that 
    \begin{align}\label{e.121506}
        \begin{split}
            &\|V(t)-\Vc(t)\|^2+\|\bV(t)-\bVc(t)\|_1^2\\
            &\quad \leq \left(\|V_0-\Vc_0\|^2+\|\bV_0-\bVc_0\|_1^2\right)\exp\left(-\nu \ll_Nt+C\int_0^t\left(1+\|V(s)\|_1^2+\|\bV(s)\|_{2}^2\right)ds\right)
        \end{split}
    \end{align}
    for all $t\geq 0$. This would verify Assumption $(\mathrm{H1})$ with  $\zeta=\nu \ll_N$, $\kappa=C$, and $U$, $S$ given by \eqref{def:USq}.

    Let $\bU=\bV-\bVc$ and $\tU=\tV-\tVc$. From \eqref{e.081201} and  \eqref{e.121501}, we see 
    \begin{align*}
        &\partial_t\bU+\Pc_h(\bV\cdot\nabla_h  \bU)+ \Pc_h(\bU\cdot\nabla_h  \bVc)=\nu\Delta_h \bU-\frac{\nu \ll_N}{2}\bP_N\bU,  \\
        & \partial_t\tU+\bV\cdot\nabla_h  \tU + \bU\cdot\nabla_h  \tVc + \frac12 \tVc^{\perp}(\nabla_h ^{\perp}\cdot \bU) + \frac12 \tU^{\perp}(\nabla_h ^{\perp}\cdot\bVc) = \nu\DD \tU-\frac{\nu \ll_N}{2}\tP_N\tU.
    \end{align*}
    The estimate for $\bU$ is the same as that for the Navier-Stokes equation in \cite{butkovsky2020generalized}. One has 
    \begin{align}\label{e.121502}
        \frac{d}{dt}\|\bU\|^2\leq  - \nu\|\nabla_h\bU\|^2- \nu \ll_N\|\bP_N\bU\|^2+\frac{C}{\nu}\|\bV\|_1^2\|\bU\|^2. 
    \end{align}
    In addition, we have 
    \begin{align*}
        \frac12\frac{d}{dt}\|\bU\|_1^2 = -\nu\|\DD_h\bU\|^2-\frac{\nu \ll_N}{2} \|\bP_N\nabla_h\bU\|^2 - \left\langle \Pc_h(\bV\cdot\nabla_h  \bU)+ \Pc_h(\bU\cdot\nabla_h  \bVc), \DD_h\bU\right\rangle.
    \end{align*}
    Since 
    \begin{align*}
        \left\langle \Pc_h(\bV\cdot\nabla_h  \bU), \DD_h\bU\right\rangle\leq C\|\bV\|_2\|\nabla_h\bU\|\|\DD_h\bU\|\leq \frac{\nu}{4}\|\DD_h\bU\|^2 + \frac{C}{\nu}\|\bV\|_2^2\|\nabla_h\bU\|^2,
    \end{align*}
    and due to $\langle \bU\cdot\nabla_h\bU,\DD_h\bU\rangle=0$ on $\mathbb T^2$, one has 
    \begin{align*}
        \left\langle  \Pc_h(\bU\cdot\nabla_h  \bVc), \DD_h\bU\right\rangle &= \left\langle  \Pc_h(\bU\cdot\nabla_h  \bV), \DD_h\bU\right\rangle\\
        &\leq C\|\nabla\bU\|\|\bV\|_2\|\DD_h\bU\|\leq \frac{\nu}{4}\|\DD_h\bU\|^2 + \frac{C}{\nu}\|\bV\|_2^2\|\nabla_h\bU\|^2.
    \end{align*}
    Therefore, we obtain
    \begin{align}\label{e.121601}
        \frac{d}{dt}\|\bU\|_1^2\leq \nu\|\DD_h\bU\|^2-\nu \ll_N \|\bP_N\nabla_h\bU\|^2 + \frac{C}{\nu}\|\bV\|_2^2\|\bU\|_1^2. 
    \end{align}
    By cancellation of the nonlinear terms, we obtain from the equation of $\tU$ that 
    \begin{align}\label{e.121503}
        \frac12\frac{d}{dt}\|\tU\|^2+\nu\|\nabla\tU\|^2 = -\langle\bU\cdot\nabla_h  \tVc, \tU\rangle - \frac12\langle \tVc^{\perp}(\nabla_h ^{\perp}\cdot \bU), \tU\rangle-\frac{\nu \ll_N}{2}\|\tP_N\tU\|^2.
    \end{align}
    Since $\bU$ is divergence free over $\mathbb T^2$ and $\tU=\tV-\tVc$, one obtains from applying H\"older's inequality, Ladyzhenskaya's inequality, and Young's inequality that
    \begin{align}\label{e.121504}
        \begin{split}
            \langle\bU\cdot\nabla_h  \tVc, \tU\rangle &=\int_{\mathbb T^1} \int_{\mathbb T^2} (\bU\cdot\nabla_h  \tV) \cdot \tU  dxdy dz \\
            &\leq \int_{\mathbb T^1}\|\bU\|_{L^4(\mathbb T^2)}\|\nabla_h  \tV\|_{L^2(\mathbb T^2)}\|\tU\|_{L^4(\mathbb T^2)} dz\\
            &\lesssim \Big(\|\bU\|\|\nabla_h \bU\|\Big)^{\frac12}\|\nabla  \tV\|\left(\int_{\mathbb T^1}\|\tU\|_{L^4(\mathbb T^2)}^2 dz\right)^{\frac12}\\
            &\lesssim \Big(\|\bU\|\|\nabla_h \bU\|\Big)^{\frac12}\|\nabla  \tV\|\left(\int_{\mathbb T^1}\|\tU\|_{L^2(\mathbb T^2)}\|\nabla_h \tU\|_{L^2(\mathbb T^2)} dz\right)^{\frac12}\\
            &\leq \Big(\|\bU\|\|\nabla_h \bU\|\Big)^{\frac12}\|\nabla  \tV\|\left(\|\tU\|\|\nabla \tU\|\right)^{\frac12}\\
            &\leq \frac{C}{\nu}\|\nabla  \tV\|^2\left(\|\bU\|^2 + \|\tU\|^2\right) + \frac{\nu}{4} \left(\|\nabla_h \bU\|^2  + \|\nabla \tU\|^2\right).
        \end{split}
    \end{align}
    Similarly, we have 
    \begin{align}\label{e.121505}
        \begin{split}
            \langle \tVc^{\perp}(\nabla_h ^{\perp}\cdot \bU), \tU\rangle& = \langle \tV^{\perp}(\nabla_h ^{\perp}\cdot \bU), \tU\rangle\\
            &\leq \|\nabla_h ^{\perp}\cdot \bU\|\left(\int_{\mathbb T^1}\|\tV\|_{L^4(\mathbb T^2)}^2dz\right)^{\frac12}\left(\int_{\mathbb T^1}\|\tU\|_{L^4(\mathbb T^2)}^2dz\right)^{\frac12}\\
            &\lesssim \|\nabla_h  \bU\|\left(\int_{\mathbb T^1}\|\tV\|_{L^2(\mathbb T^2)}\|\nabla_h  \tV\|_{L^2(\mathbb T^2)}dz\right)^{\frac12}\left(\int_{\mathbb T^1}\|\tU\|_{L^2(\mathbb T^2)}\|\nabla_h \tU\|_{L^2(\mathbb T^2)}dz\right)^{\frac12}\\
            &\leq \|\nabla_h  \bU\|\left(\|\tV\|\|\nabla  \tV\|\|\tU\|\|\nabla \tU\|\right)^{\frac12}\\
            &\leq C\|\tV\|\|\bU\|_1^2+ \frac{C}{\nu}\|\nabla \tV\|^2\|\tU\|^2+\frac{\nu}{4}\|\nabla \tU\|^2.
        \end{split}
    \end{align}
    Observe that by Poincar\'e's inequality we have
    \begin{align*}
        &\nu \ll_N\|\bV\|^2\leq \nu \ll_N\|\bP_N\bV\|^2+\nu\|\nabla_h \bV\|^2\quad\text{and}\quad\nu \ll_N\|\tV\|^2\leq \nu \ll_N\|\tP_N\tV\|^2+\nu\|\nabla \tV\|^2.
    \end{align*}
    Combining this fact with \eqref{e.121502}--\eqref{e.121505} we obtain 
    \begin{align*}
        &\frac{d}{dt}\left(\|\bU\|^2+\|\tU\|^2+\|\bU\|_1^2\right)\notag
        \\
        &\leq \left(-\nu \ll_N+C
        (1+\|V\|_1^2+\|\bV\|_{2}^2)\right)\left(\|\bU\|^2+\|\tU\|^2+\|\bU\|_1^2\right).
    \end{align*}
    The desired estimate \eqref{e.121506} then follows from Gr\"onwall's inequality. 

    \subsubsection*{Verifying Assumption $(\mathrm{H2})$.} 
    Observe that by It\^o's formula and the equation for $\bV$, we have 
    \begin{align*}
        \|\bV(t)\|_1^2&=\|\bV(0)\|_1^2\notag
        \\
        &\quad+ 2\int_0^t\left\langle \nu\DD_h\bV-\Pc\left(\bV\cdot\nabla_h\bV\right), \DD_h\bV\right\rangle ds+ 2\int_0^t\left\langle \DD_h\bV(s), {\bP}GdW(s)\right\rangle+\|{\bP}G\|_{L_2(H,H^1)}^2t.
    \end{align*}
    Since $\left\langle \Pc\left(\bV\cdot\nabla_h\bV\right), \DD_h\bV\right\rangle =0$, we may combine this energy estimate with \eqref{e.121507} to obtain 
    \begin{align}\label{e.121603}
        \|V(t)\|^2+\|\bV(t)\|_1^2&+2\nu\int_0^t\left(1+\|V(s)\|_1^2+\|\bV(s)\|_2^2\right)ds\leq \|V_0\|^2+\|\bV_0\|_1^2+\mathcal{G}^2 t + M(t),
    \end{align}
    where $\mathcal{G}^2 = 2\nu+\|{\bP}G\|_{L_2(H,H^1)}^2 + \|G\|_{L_2(H,H)}^2$, and 
        \[
            M(t) = 2\int_0^t\left\langle \DD_h\bV(s), {\bP}GdW(s)\right\rangle + 2\int_0^t\left\langle V(s), GdW(s)\right\rangle,
        \]
    is the martingale term whose quadratic variation satisfies
    \begin{align*}
        [M](t)\leq 2\|G\|_{L_2(H,H)}^2\int_0^t\left(\|\bV(s)\|_2^2+\|V(s)\|^2\right)ds.
    \end{align*}
    Therefore, Assumption $(\mathrm{H2})$ is verified with $\mu=2\nu$, $b = \mathcal{G}^2$, $b_1=2\|G\|_{L_2(H,H)}^2$ and $b_2=0$. Recall from Assumption $(H1)$ that $\zeta=\nu \ll_N$ and $\kappa=C$. To ensure $\zeta>\kappa\frac{b}{\mu}$, we choose $N$ sufficiently large such that 
    \begin{align*}
        \nu\ll_N>C\frac{\mathcal{G}^2}{2\nu}.
    \end{align*}
    This verifies that $(H2)$ holds with $U,S$ given by \eqref{def:USq}.

\subsubsection*{Verifying Assumption $(\mathrm{H3})$.} We verify this assumption by using \cite[Lemma 4.1]{butkovsky2020generalized}. Let 
   \begin{align*}
      &d\overline{W}(t): = dW(t)-\frac{\nu \ll_N}{2}({\bP}G)^{-1}\bP_N(\bV-\bVc)dt,
      \\
      &d\widetilde{W}(t): = dW(t)-\frac{\nu \ll_N}{2}((I-\bP)G)^{-1}\tP_N(\tV-\tVc)dt,
   \end{align*}
   where the pseudo-inverses $({\bP}G)^{-1}$ and $((I-\bP)G)^{-1}$ are well-defined and bounded by the assumption \eqref{e.121602}. Therefore
   \begin{align*}
    \frac{\nu \ll_N}{2}\left(\|({\bP}G)^{-1}\bP_N(\bV-\bVc)\|+\|((I-\bP)G)^{-1}\tP_N(\tV-\tVc)\|\right)\leq C\|V-\Vc\|,
   \end{align*}
   which verifies Assumption $(\mathrm{H3})$.

   It remains to find the desired Lyapunov function $F$. By choosing $F(V) = \|V\|^2+\|\bV\|_1^2$, we infer from \eqref{e.121603} and Poincar\'e's inequality that 
   \begin{align*}
    \Eb F(V_t)\leq F(V_0)-2\nu C\Eb\int_0^tF(V_r)dr+\mathcal{G}^{2}t,
   \end{align*}
   verifying the inequality required in Theorem \ref{thm:butkovsky}. On the other hand, one readily sees that the functions $U$ and $q$ from \eqref{def:USq} are bounded on the level sets $\left\{F\leq M\right\}$ and $\left\{F\leq M\right\}\times \left\{F\leq M\right\}$, for any $M>0$. Lastly, one has $\|u-v\|\leq q(u,v)^{\frac12}$. We therefore deduce from \cref{thm:butkovsky} that the system \eqref{e.081201} has a unique invariant measure and that a spectral gap exists.
\end{proof}

The next step is to bootstrap convergence to $H^1$ by exploiting the availability of higher-order moments afforded by parabolic regularization effects. To this end, we establish two key lemmas. For the first lemma, we establish H\"older continuity of the Markov kernel $\{P_t\}_{t\geq0}$ (\cref{l.120301}). For the second lemma, we make use of the first lemma to assert the bootstrap step (\cref{l.121701}).

\begin{lemma}\label{l.120301}
    Let $R>0$ and $v_1,v_2\in B_R(H)$, the ball in $H$ with radius $R$ centered at the origin. Then for any $\gg\in(0,1]$, there is $\dd<\gg$ such that for any $\phi\in C^{\gamma}(H^1)$, $t>0$, 
    \[|P_t\phi(v_1)-P_t\phi(v_2)|\leq C(t)\|\phi\|_{C^{\gg}(H^1)}\exp(\eta_0 R^2)\|v_1-v_2\|^{\dd}\]
    where $\eta_0$ is the constant from \cref{l.120101} and $C(t)$ depends on $\|v_1\|, \|v_2\|$ and is bounded on any interval $[t_1,t_2]\subset(0,\infty)$. 
\end{lemma}
\begin{proof}
    Let $V_i:=(\bV_i,\tV_i)$, for $i=1,2$, be two solutions of equation \eqref{e.081201} and $\overline{U} = \bV_1-\bV_2, \widetilde{U}=\tV_1-\tV_2$ be their difference. Then $\overline{U}$ satisfies the equation 
    \begin{align*}
        \partial_t\overline{U}+\Pc_h(\bV_1\cdot\nabla_h  \overline{U})+ \Pc_h(\overline{U}\cdot\nabla_h  \bV_2)=\nu\Delta \overline{U},  
    \end{align*}
    from which we deduce 
    \begin{align}\label{e.120101}
        \frac12\frac{d}{dt}\|\overline{U}\|^2 + \nu\|\nabla_h \overline{U}\|^2 = -\langle\overline{U}\cdot\nabla_h  \bV_2, \overline{U}\rangle.
    \end{align}
    Combining this with the estimate
    \begin{align}\label{e.120103}
        \langle\overline{U}\cdot\nabla_h  \bV_2, \overline{U}\rangle\leq C\|\nabla_h  \bV_2\|^2\|\overline{U}\|^2+\frac{\nu}{4}\|\nabla_h  \overline{U}\|^2
    \end{align} 
    and Gr\"onwall's inequality we obtain 
    \begin{align}\label{e.120104}
        \|\overline{U}(t)\|^2+\int_0^t\|\nabla_h  \overline{U}(s)\|^2ds\leq \|\overline{U}(0)\|^2\exp\left(C\int_0^t\|\nabla_h  \bV_2(s)\|^2ds\right).
    \end{align}

    Note that $\widetilde{U}$ satisfies 
    \begin{align*}
        \partial_t\widetilde{U}+\bV_1\cdot\nabla_h  \widetilde{U} + \overline{U}\cdot\nabla_h  \tV_2 + \frac12 \tV_1^{\perp}(\nabla_h ^{\perp}\cdot \overline{U}) + \frac12 \widetilde{U}^{\perp}(\nabla_h ^{\perp}\cdot\bV_2) = \nu\DD \widetilde{U}.
    \end{align*}
    Therefore, one has 
    \begin{align}\label{e.120102}
        \frac12\frac{d}{dt}\|\widetilde{U}\|^2 + \nu\|\nabla \widetilde{U}\|^2 = -\langle\overline{U}\cdot\nabla_h  \tV_2, \widetilde{U}\rangle - \frac12\langle \tV_{1}^{\perp}(\nabla_h ^{\perp}\cdot \overline{U}), \widetilde{U}\rangle.
    \end{align}
    The same argument as in estimating \eqref{e.121504} and \eqref{e.121505} yields
    \begin{align*}
        \langle\overline{U}\cdot\nabla_h  \tV_2, \widetilde{U}\rangle \leq C\|\nabla  \tV_2\|^2\left(\|\overline{U}\|^2 + \|\widetilde{U}\|^2\right) + \frac{\nu}{4} \left(\|\nabla_h \overline{U}\|^2  + \|\nabla \widetilde{U}\|^2\right),
    \end{align*}
    and 
    \begin{align*}
        \langle \tV_{1}^{\perp}(\nabla_h ^{\perp}\cdot \overline{U}), \widetilde{U}\rangle\leq  \|\nabla_h  \overline{U}\|^2\|\tV_{1}\|+ C\|\nabla \tV_{1}\|^2\|\widetilde{U}\|^2+\frac{\nu}{4}\|\nabla \widetilde{U}\|^2.
    \end{align*}
    Combining these two estimates with \eqref{e.120101}, \eqref{e.120103} and \eqref{e.120102} we obtain 
    \begin{align*}
        \frac{d}{dt}\left(\|\overline{U}\|^2+\|\widetilde{U}\|^2\right)\leq C\left(\|\nabla_h  \bV_2\|^2+\|\nabla  \tV_2\|^2 + \|\nabla \tV_1\|^2\right)\left(\|\overline{U}\|^2+\|\widetilde{U}\|^2\right) + \|\nabla_h  \overline{U}\|^2\|\tV_{1}\|.
    \end{align*}
    Denote 
    \[
        C(s,t) = \exp\left(C\int_{s}^t\left(\|\nabla_h  \bV_2(r)\|^2+\|\nabla \tV_2(r)\|^2 + \|\nabla \tV_1(r)\|^2 \right) dr\right).
    \]
    By Gr\"onwall's inequality and \eqref{e.120104}, we deduce 
    \begin{align}\label{e.120105}
            \|\overline{U}(t)\|^2+\|\widetilde{U}(t)\|^2&\leq \left(\|\overline{U}(0)\|^2+\|\widetilde{U}(0)\|^2\right)C(0,t) + \int_0^t \|\nabla_h  \overline{U}(s)\|^2\|\tV_{1}(s)\|C(s,t)ds\notag
            \\
            &\leq\left(\|\overline{U}(0)\|^2+\|\widetilde{U}(0)\|^2\right)C(0,t)\left(1+\left(\sup_{s\in[0,t]}\|\tV_{1}(s)\|\right) \int_0^t \|\nabla_h  \overline{U}(s)\|^2ds\right)\notag
            \\
            &\leq \left(\|\overline{U}(0)\|^2+\|\widetilde{U}(0)\|^2\right)C(0,t)\left(1+\sup_{s\in[0,t]}\|\tV_{1}(s)\|\right). 
    \end{align}

    Now let $\phi$ be a function from $C^{\gg}(H^1)$ such that $\|\phi\|_{C^\gamma}\leq 1$ and let $v_1,v_2\in H$ denote the initial conditions corresponding to $V_1,V_2$, respectively. Then for any $t>0$, and $\gg_0\leq \gg$, by estimate \eqref{e.120105}, we have 
    \begin{align*}
        |P_t\phi(v_1)-P_t\phi(v_2)| &= \left|\Eb \phi(V_1(t))-\Eb \phi(V_2(t))\right|
        \\
        &\leq C\Eb\|V_1(t)-V_2(t)\|_1^{\gg_0}
        \\
        &\leq C\Eb\|V_1(t)-V_2(t)\|^{\frac{\gg_0}{2}}\left(\|V_1(t)\|_{2}^{\frac{\gg_0}{2}}+\|V_2(t)\|_{2}^{\frac{\gg_0}{2}}\right)
        \\
        &\leq \|v_1-v_2\|^{\frac{\gg_0}{2}}\Eb C_2(t)\left(\|V_1(t)\|_{2}^{\frac{\gg_0}{2}}+\|V_2(t)\|_{2}^{\frac{\gg_0}{2}}\right),
    \end{align*}
    where 
        \[
            C_2(t) = \left(1+\sup_{s\in[0,t]}\|V_1(s)\|^{\gg_0/4}\right)\exp\left(C\gg_0\int_0^t\|\nabla V_2(s) \|^2 + \|\nabla V_1(s) \|^2 ds\right).
        \]
    Now in view of \cref{l.120101}, by H\"older's inequality we have 
    \begin{align*}
        &\Eb C_2(t)\left(\|V_1(t)\|_{2}^{\frac{\gg_0}{2}}+\|V_2(t)\|_{2}^{\frac{\gg_0}{2}}\right)\\
        &\leq C\left(1+\Eb\sup_{s\in[0,t]}\|V_1(s)\|^{4\gg_0}+\Eb \exp\left(4C\gg_0\int_0^t\|\nabla  V_2(s)\|^2 + \|\nabla  V_1(s)\|^2ds\right)+\Eb\|V_1(t)\|_{2}^{\gg_0}+\Eb\|V_2(t)\|_{2}^{\gg_0}\right)\\
       &\leq C(t)\exp(\eta_0 R^2), 
    \end{align*} 
    by taking $\gamma_0$ sufficiently small so that $4C\gg_0\leq \eta_0\nu$ for $\eta_0$ from \cref{l.120101}. The proof is complete by choosing $\delta = \gamma_0/2$.
\end{proof}

The next lemma asserts the bootstrap step from exponential mixing with respect to $C^\gg(H)$ observables to exponential mixing with respect to $C^1(H^1)$ observables.

\begin{lemma}\label{l.121701}
    If for any $\gg\in (0,1]$, $\phi\in C^{\gg}(H)$ and $v_0\in H^1$, one has 
    \begin{align}\label{e.022801}
        \left|P_{t}\phi(v_0)-\int_{H}\phi(v)\mu(dv)\right|\leq C(\|v_0\|_1)\|\phi\|_{C^{\gg}(H)}e^{-c t}, 
    \end{align}
    for all $t\geq0$, then one additionally has
    \begin{align*}
        \left|P_{t}\varphi(v_0)-\int_{H}\varphi(v)\mu(dv)\right|\leq C(\|v_0\|_1)\|\varphi\|_{C^{1}(H^1)}e^{-\frac12ct},
    \end{align*}
    for all $t\geq0$ and $\varphi\in C^{1}(H^1)$. 
\end{lemma}

\begin{proof}
    Pick any  $\varphi\in C^{1}(H^1)$ and denote $\phi = P_1\varphi$. By the invariance of $\mu$, for $t\geq 1$ we have 
    \begin{align*}
        \left|P_{t}\varphi(v_0)-\int_{H}\varphi(v)\mu(dv)\right|= \left|P_{t-1}\phi(v_0)-\int_{H}\phi(v)\mu(dv)\right|.
    \end{align*}
    Fix $\delta$ from Lemma \ref{l.120301}. Let $\chi_R\in C^{\dd}(H)$ such that $0\leq \chi_R\leq 1 $ and $\chi_R(u)=1$ for $\|u\|\leq R$, and $\chi_R(u)=0$ for $\|u\|\geq R+1$. We also denote $\overline{\chi}_R = 1- \chi_R$. Then by \cref{l.120301}, we have $\phi\chi_R\in C^{\dd}(H)$ with 
    \[\|\phi\chi_R\|_{C^{\dd}(H)}\leq C\|\varphi\|_{C^1(H^1)}\exp\left(\eta_0 R^2\right).\]
    By applying \eqref{e.022801} for $\phi\chi_R$,
    we have
    \begin{align*}
        E_R: = \left|P_{t-1}(\phi\chi_R)(v_0)-\int_{H}(\phi\chi_R)(v)\mu(dv)\right|\leq C(\|v_0\|_1)\|\varphi\|_{C^1(H^1)}\exp\left(\eta_0 R^2-ct\right). 
    \end{align*}
    Note that 
    \begin{align*}
        \left|P_{t-1}\phi(v_0)-\int_{H}\phi(v)\mu(dv)\right|\leq  E_R + \overline{E}_R,
    \end{align*}
    with 
    \begin{align*}
        \overline{E}_R = \left|P_{t-1}(\phi\overline{\chi}_R)(v_0)-\int_{H}(\phi\overline{\chi}_R)(v)\mu(dv)\right|.
    \end{align*}
    Note that $\sup_{u\in H}|\phi(u)|\leq \|\varphi\|_{C^1(H^1)}$. By Markov's inequality and estimate \eqref{e.120308}, we deduce
    \begin{align*}
        \overline{E}_R &\leq \int_{H}\phi(u)\overline{\chi}_R(u)P_{t-1}(v_0,du)+ \sup_{u\in H}|\phi(u)|\mu\{\|v\|\geq R\}\\
        &\leq \|\varphi\|_{C^1(H^1)}\left(\int_{\{\|u\|\geq R\}}P_{t-1}(v_0,du)+e^{-\eta_0 R^2}\int_{H}\exp(\eta_0\|v\|^2)\mu(dv)\right)\\
        &\leq \|\varphi\|_{C^1(H^1)}e^{-\eta_0 R^2}\left(\Eb\exp\left( \eta_0\|V(t-1;v_0)\|^2\right)+\int_{H}\exp(\eta_0\|v\|^2)\mu(dv)\right)\\
        &\leq C\|\varphi\|_{C^1(H^1)}e^{-\eta_0 R^2}\exp\left(\eta_0\|v_0\|^2\right).
    \end{align*}
    Upon making the choice $R = \sqrt{\frac{ct}{2\eta_0}}$, we therefore conclude
    \begin{align*}
        \left|P_{t-1}\phi(v_0)-\int_{H}\phi(v)\mu(dv)\right|&\leq \|\varphi\|_{C^1(H^1)}\left(C(\|v_0\|_1)e^{\eta_0 R^2-ct} + C(\|v_0\|)e^{-\eta_0 R^2} \right)\\
        & = C(\|v_0\|_1)\|\varphi\|_{C^1(H^1)} e^{-\frac12ct},
    \end{align*}
    as desired. Note that the inequality extends to $t\in[0,1]$ by boundedness.
\end{proof}

The proof of \cref{t.121701} now readily follows.

\begin{proof}[Proof of Theorem \ref{t.121701}]
Given any $\gg\in (0,1]$, $\phi\in C^{\gg}(H)$, and $v_0\in H^1$, it follows from \cref{c:081201} that
    \begin{align}
         \left|P_{t}\phi(v_0)-\int_{H}\phi(v)\mu(dv)\right|\leq C\|\phi\|_{C^{\gg}(H)}C_\gg(\|v_0\|_1)e^{-c\gg t},\notag
    \end{align}
holds for all $t\geq0$. By \cref{l.121701}, we then conclude that
     \begin{align*}
        \left|P_{t}\varphi(v_0)-\int_{H}\varphi(v)\mu(dv)\right|\leq C(\|v_0\|_1)\|\varphi\|_{C^{1}(H^1)}e^{-\frac12c\gg t},
    \end{align*}
for all $t\geq0$ and $\varphi\in C^{1}(H^1)$, as claimed.

\end{proof}

\section{Proof of the Main Theorem}\label{sect:final}

We are finally ready to prove our main result of this article, \cref{t.121702}, by combining \cref{thm:convergence_C} and \cref{t.121701}.

\begin{proof}[Proof of \cref{t.121702}]
    Recall that 
        \[
            \mathcal{U}_t(\bv,\tv) = (\bv,e^{\aa Jt}\tv)=(\bv,\tu).
        \]
    Let $\ee>0$ and $\varphi\in C^1(H^1)$. By \cref{t.121701}, given $v_0\in H^1$, there exists $T_0>0$ such that 
    \begin{align*}
        \left|P_{t}\varphi(v_0)-\int_{H^1}\varphi(v)\mu(dv)\right|<\ee/2,
    \end{align*}
    for all $t\geq T_0$. On the other hand, by \cref{thm:convergence_C}, for each $t\geq T_0$, there is $\aa_0=\aa_0(t,\ee)>0$ such that
    \begin{align*}
        \left|\int_{H^1}\varphi(v)\mathcal{U}_t^*P_{t}^{\aa}(v_0,dv)-P_{t}\varphi(v_0)\right|<\ee/2, 
    \end{align*}
    for all $\aa\geq \aa_0$. Therefore, for all $t\geq T_0$ and $\alpha\geq \alpha_0(t,\ee)$, we have 
    \begin{align*}
        \left|\int_{H^1}\varphi(v)\mathcal{U}_t^*P_{t}^{\aa}(v_0,dv)-\int_{H^1}\varphi(v)\mu(dv)\right|<\ee. 
    \end{align*}
    Consequently, we have the weak convergence 
    \[\lim_{t\to\infty}\lim_{\aa\to\infty}\mathcal{U}_t^*P_{t}^{\aa}(v_0,\cdot)=\mu\]
    as desired. 
\end{proof}

\subsection*{Acknowledgments}
Q.L. was partially supported by an AMS-Simons travel grant. V.R.M. was in part supported by the National Science Foundation through DMS 2213363 and DMS 2206491, as well as the Mary P. Dolciani Halloran Foundation.

\appendix 
\section{Auxiliary lemmas}
\begin{lemma}[\hspace{-0.01pt}\cite{constantin2015long}]\label{l.080801}
    Let $s>0, p\in(1,\infty)$ and $f,g\in C^{\infty}(\mathbb T^d)$ with $d=2,3$. Then 
    \begin{equation*}
        \begin{split}
            \|\LL^{s}(fg)\|_{L^p}\lesssim \|g\|_{L^{p_1}}\|\LL^sf\|_{L^{p_2}}+\|\LL^sg\|_{L^{p_3}}\|f\|_{L^{p_4}},
        \end{split}
    \end{equation*}
    where $1/p=1/p_1+1/p_2=1/p_3+1/p_4$ and $p_2,p_3\in (1,\infty)$. Moreover, 
    \begin{align*}
        \left\|\Lambda^s(f g)-f \Lambda^s g\right\|_{L^p} \lesssim \|\nabla  f\|_{L^{p_1}}\left\|\Lambda^{s-1} g\right\|_{L^{p_2}}+ \left\|\Lambda^s f\right\|_{L^{p_3}}\|g\|_{L^{p_4}}
    \end{align*}
    where $p_i$ are as above.  
\end{lemma}

\begin{lemma}[\hspace{-0.01pt}\cite{temam2024navier}]\label{l.081001}
    Let $b(u,v,w)=\langle u\cdot \nabla  v, w \rangle$ for $u,v,w \in C^{\infty}(\Tb^d;\Rb^d)$. Then 
    \[|b(u,v,w)|\leq C\|u\|_{m_1}\|v\|_{m_2+1}\|w\|_{m_3},\]
    where 
    \begin{align*}
        &m_1+m_2+m_3\geq \frac{d}{2}, \, \text{ if } m_i\neq \frac{d}{2} \text{ for all } i=1,2,3, \\
        &m_1+m_2+m_3>\frac{d}{2}, \, \text{ if } m_i=\frac{d}{2} \text{ for some } i. 
    \end{align*}
\end{lemma}

\begin{lemma}[\hspace{-0.05pt}\cite{flandoli2012stochastic}]\label{l.081101}
    Let $\aa>0$, $T>0$, $p\geq 1$, $\ll\neq 0$, and $f(t)$ be a scalar valued stochastic process such that $\sup_{t\in[0,T]}|f(t)|$ is measurable. If for some $\bb>0$, 
    \begin{align*}
        \Eb\sup_{t\in[0,T]}|f(t)|^p &\leq C_0,
        \\
        \Eb|f(t)-f(s)|^p&\leq C_0|t-s|^{\bb}, \quad s,t\in [0,T], 
    \end{align*}
    then there is $C=C(C_0,T,p)$ such that 
    \begin{align*}
        \Eb\sup_{t\in[0,T]}\left|\int_0^te^{\mi \aa \ll s}f(s)ds\right|^p\leq C\left(\left|\frac{1}{\aa\ll}\right|^{\bb}+\left|\frac{1}{\aa\ll}\right|\right)
    \end{align*}
\end{lemma}

\begin{lemma}\label{l.081002}
    Let $u,v\in C^{\infty}(\mathbb R^d,\mathbb R^2)$, $\nabla_h  = (\pp_x,\pp_y)$ and $J=\left(\begin{array}{cc}0 & -1 \\ 1 & 0\end{array}\right)$. Then for $\aa,\bb\in \mathbb R$, 
    \begin{align*}
        (e^{\aa J} u)\cdot\nabla_h (e^{\bb J}v)=\frac12 e^{(\aa+\bb)J}\left(u\cdot\nabla_h  v- u^{\perp}\cdot\nabla_h  v^{\perp}\right) + \frac12 e^{(\bb-\aa)J}\left(u\cdot\nabla_h  v- u\cdot\nabla_h ^{\perp} v^{\perp}\right)
    \end{align*}
    and 
    \begin{align*}
        \left(\nabla_h  \cdot (e^{\aa J} u)\right)(e^{\bb J}v)=\frac12 e^{(\aa+\bb)J}\left((\nabla_h  \cdot u) v-(\nabla_h  \cdot u^{\perp}) v^{\perp}\right) + \frac12 e^{(\bb-\aa)J}\left((\nabla_h  \cdot u) v-(\nabla_h ^{\perp} \cdot u) v^{\perp}\right). 
    \end{align*}
    In particular, when $d=3$ we have 
    \begin{align}\label{e.081105}
       w\left(e^{-\aa J}u\right)\partial_{z}u=  \frac12 e^{-\aa J}\left(w(u) \pp_zu-w(u^{\perp}) \pp_zu^{\perp}\right) + \frac12 e^{\aa J}\left(w(u) \pp_zu+w(u^{\perp}) \pp_zu^{\perp}\right),
    \end{align}
    where $w$ is defined as in \eqref{e.081104}. 
\end{lemma}
\begin{proof}
    Let $\phi$ be the linear isomorphism of $\mathbb{C}=\mathrm{span}\{1,\mi\}$ and $\mathbb R^2=\mathrm{span}\{(1,0),(0,1)\}$ such that 
    \begin{align*}
        \phi((v_1,v_2)) = v_1+\mi v_2
    \end{align*}
    Under this isomorphism, since $\phi(Jv) = i\phi(v)$ for $v=(v_1,v_2)$, we have 
    \begin{align*}
        \phi(e^{\aa J}v) = e^{\mi \aa}\phi(v). 
    \end{align*}
    Note that $u\cdot v =\Re(\phi(u)\overline{\phi(v)}) =\frac12(\phi(u)\overline{\phi(v)}+\overline{\phi(u)}\phi(v))$. It follows that 
    \begin{align*}
       \phi((e^{\aa J} u)\cdot\nabla_h (e^{\bb J}v))&=(e^{\aa J} u)\cdot\nabla_h  \phi(e^{\bb J}v)\\
       & = \frac12(e^{\mi\aa}\phi(u)(\pp_x-\mi \pp_y)+e^{-\mi\aa}\overline{\phi(u)}(\pp_x+\mi \pp_y))e^{\mi\bb}\phi(v).
    \end{align*}
    It remains to identify the real and imaginary parts. We have 
    \begin{align*}
        &(e^{\mi\aa}\phi(u)(\pp_x-\mi \pp_y))e^{\mi\bb}\phi(v)\\
        & = e^{\mi(\aa+\bb)}\left(u_1\pp_x+u_2\pp_y+\mi(u_2\pp_x-u_1\pp_y)\right)(v_1+\mi v_2)\\
        & = e^{\mi(\aa+\bb)}\left(u\cdot\nabla_h  v_1+\mi u\cdot\nabla_h  v_2 + \mi(u_2\pp_xv_1-u_1\pp_yv_1)-(u_2\pp_xv_2-u_1\pp_yv_2)\right)\\
        & = e^{\mi(\aa+\bb)}\left(u\cdot\nabla_h  v_1+\mi u\cdot\nabla_h  v_2 - \mi u^{\perp}\cdot\nabla_h  v_1+u^{\perp}\cdot\nabla_h  v_2\right), 
    \end{align*}
    and 
    \begin{align*}
        &(e^{-\mi\aa}\overline{\phi(u)}(\pp_x+\mi \pp_y))e^{\mi\bb}\phi(v)\\
        & = e^{\mi(\bb-\aa)}\left(u_1\pp_x+u_2\pp_y+\mi(u_1\pp_y-u_2\pp_x)\right)(v_1+\mi v_2)\\
        & = e^{\mi(\bb-\aa)}\left(u\cdot\nabla_h  v_1+\mi u\cdot\nabla_h  v_2 +\mi (u_1\pp_y-u_2\pp_x)v_1- (u_1\pp_y-u_2\pp_x)v_2\right)\\
        & = e^{\mi(\bb-\aa)}\left(u\cdot\nabla_h  v_1+\mi u\cdot\nabla_h  v_2-\mi u\cdot\nabla_h ^{\perp}v_1+u\cdot\nabla_h ^{\perp}v_2\right). 
    \end{align*}
    By converting back to vectors through $\phi^{-1}$ we have 
    \begin{align*}
        (e^{\aa J} u)\cdot\nabla_h (e^{\bb J}v)& = \phi^{-1}\phi((e^{\aa J} u)\cdot\nabla_h (e^{\bb J}v))\\
        & = \frac12\phi^{-1}\left(e^{\mi(\aa+\bb)}\left(u\cdot\nabla_h  v_1+\mi u\cdot\nabla_h  v_2 - \mi u^{\perp}\cdot\nabla_h  v_1+u^{\perp}\cdot\nabla_h  v_2\right)\right)\\
        &\quad +\frac12\phi^{-1}\left(e^{\mi(\bb-\aa)}\left(u\cdot\nabla_h  v_1+\mi u\cdot\nabla_h  v_2-\mi u\cdot\nabla_h ^{\perp}v_1+u\cdot\nabla_h ^{\perp}v_2\right)\right)\\
        & =\frac12 e^{(\aa+\bb)J}\left(u\cdot\nabla_h  v- u^{\perp}\cdot\nabla_h  v^{\perp}\right) + \frac12 e^{(\bb-\aa)J}\left(u\cdot\nabla_h  v- u\cdot\nabla_h ^{\perp} v^{\perp}\right).
    \end{align*}
    Similarly, we have 
    \begin{align*}
        \phi\left(\left(\nabla_h  \cdot (e^{\aa J} u)\right)(e^{\bb J}v)\right) &= \left(\nabla_h  \cdot (e^{\aa J} u)\right)\phi(e^{\bb J}v)\\
        & = \frac12((\pp_x-\mi \pp_y)e^{\mi\aa}\phi(u)+(\pp_x+\mi \pp_y)e^{-\mi\aa}\overline{\phi(u)})e^{\mi\bb}\phi(v).
    \end{align*}
    Observe that 
    \begin{align*}
        &((\pp_x-\mi \pp_y)e^{\mi\aa}\phi(u))e^{\mi\bb}\phi(v)\\
        & = e^{\mi(\aa+\bb)}\left(\pp_xu_1+\pp_yu_2+\mi(\pp_xu_2-\pp_yu_1)\right)(v_1+\mi v_2)\\
        & = e^{\mi(\aa+\bb)}\left(\nabla_h \cdot u v_1 + \mi\nabla_h \cdot u v_2+\mi(\pp_xu_2-\pp_yu_1)v_1-(\pp_xu_2-\pp_yu_1)v_2\right)\\
        & = e^{\mi(\aa+\bb)}\left(\nabla_h \cdot u v_1 + \mi\nabla_h \cdot u v_2-\mi v_1\nabla_h \cdot u^{\perp}+v_2\nabla_h \cdot u^{\perp}\right)
    \end{align*}
    and 
    \begin{align*}
        &((\pp_x+\mi \pp_y)e^{-\mi\aa}\overline{\phi(u)})e^{\mi\bb}\phi(v)\\
        & = e^{\mi(\bb-\aa)}\left(\pp_xu_1+\pp_yu_2+\mi(\pp_yu_1-\pp_xu_2)\right)(v_1+\mi v_2)\\
        & = e^{\mi(\bb-\aa)}\left(\nabla_h \cdot u v_1 + \mi\nabla_h \cdot u v_2+\mi(\pp_yu_1-\pp_xu_2)v_1-(\pp_yu_1-\pp_xu_2)v_2\right)\\
        & = e^{\mi(\bb-\aa)}\left(\nabla_h \cdot u v_1 + \mi\nabla_h \cdot u v_2 - \mi v_1\nabla_h ^{\perp}\cdot u+v_2\nabla_h ^{\perp}\cdot u\right). 
    \end{align*}
    The desired identity follows by converting back to vectors through $\phi^{-1}$. By letting \[u(x,y,z) \text{ to be } -\int_0^zu(x,y,\zeta)d\zeta, \text{ and } v =\partial_{z}u\]
    and noting $\nabla_h ^{\perp} \cdot u = -\nabla_h \cdot u^{\perp}$, we also obtain \eqref{e.081105}.

\end{proof}

\section{Bounds on the solutions}
\begin{lemma}[$L^2$ estimates of the primitive equations]\label{l.080901}
    Let $v$ be the strong solution to \eqref{PE-system} with deterministic initial data $v_0\in H^1$ and  $T>0$. Then the following estimates hold.  
    \begin{itemize}
        \item[(i)]  For any $p\geq 2$,
        \begin{align}\label{e.081002}
            \Eb\left[\sup_{t\in[0,T]}\|v(t)\|^p+\left(\int_0^T\|v(r)\|_{1}^2dr\right)^p\right]\leq C(p,\sigma,T,\|v_0\|).
        \end{align}
        \item[(ii)] Let $\tu=e^{\aa Jt}\tv$. For any $p\geq 2$ there is a $q>0$ such that for any $s,t\in[0,T]$ and $\phi\in C^{\infty}$, one has 
        \begin{align*}
            \Eb\left|\langle \bv(t), \phi\rangle-\langle \bv(s), \phi\rangle\right|^p\leq C |t-s|^q,
        \end{align*}
        and  
        \begin{align*}
            \Eb\left|\langle \tu(t), \phi\rangle-\langle \tu(s), \phi\rangle\right|^p\leq C |t-s|^q, 
        \end{align*}
        where $C$ is independent of $\aa$. 
    \end{itemize}
\end{lemma}
\begin{proof}
    We first prove (i). By It\^o formula, the divergence free property of $(v,w)$ and integration by parts we have 
    \begin{align*}
        &\|v(t)\|^p\notag
        \\
        &=\|v(0)\|^p\notag   - p\int_0^t\langle v\cdot \nabla_h  v+w\pp_z v+\aa v^{\perp}+\nabla_h  P,v \rangle \|v\|^{p-2} dr + \frac{p}{2}\int_0^t\|\sigma\|_{L_2(H,H)}^2\|v\|^{p-2}dr\\
        & \quad -p\nu\int_0^t\|v\|_1^2\|v\|^{p-2}dr + \frac{p(p-2)}{2}\int_0^t\|v\|^{p-4}\|\sigma v\|^2 dr + p\int_0^t\|v\|^{p-2}\left\langle v(r),\sigma dW(r)\right\rangle  \\
        & = \|v(0)\|^p -p\nu\int_0^t\|v\|_1^2\|v\|^{p-2}dr + \frac{p}{2}\int_0^t\|\sigma\|_{L_2(H,H)}^2\|v\|^{p-2}dr\\
        &\quad + \frac{p(p-2)}{2}\int_0^t\|v\|^{p-4}\|\sigma v\|^2 dr + p\int_0^t\|v\|^{p-2}\left\langle v(r),\sigma dW(r)\right\rangle . 
    \end{align*}
    Therefore
    \begin{align*}
        \|v(t)\|^p + p\nu\int_0^t\|v(r)\|_1^2\|v(r)\|^{p-2}dr &\leq \|v(0)\|^p+C(p,\sigma,T) + \int_0^t\|v(r)\|^{p}dr\\
        &\quad +p\int_0^t\|v(r)\|^{p-2}\left\langle v(r),\sigma dW(r)\right\rangle . 
    \end{align*}
    By the Burkholder-Davis-Gundy inequality, we have 
    \begin{align*}
        \Eb\sup_{s\in[0,s]}\left|p\int_0^t\|v(r)\|^{p-2}\langle v(r),\sigma dW(r)\rangle \right|
        &\leq C \Eb\left(\int_0^t\|v(r)\|^{2p-2}\|\sigma\|_{L_2(H,H)}^2dr\right)^{1/2}\\
        &\leq C(p,\sigma,T)+\frac12\Eb \sup_{r\in[0,t]}\|v(r)\|^p. 
    \end{align*}
    Hence 
    \begin{align*}
        \Eb \sup_{r\in[0,t]}\|v(t)\|^p\leq C\|v(0)\|^p+C(p,\sigma,T)+C\Eb\int_0^t\|v(r)\|^{p}dr,
    \end{align*}
    which implies by Gr\"onwall's inequality that 
    \begin{align*}
        \Eb \sup_{r\in[0,t]}\|v{(t)}\|^p\leq C(p,\sigma,T,\|v_0\|). 
    \end{align*}
    Note that we have the energy equality 
    \begin{align*}
        \|v(t)\|^2+2\nu\int_0^t\|v(r)\|_1^2dr=\|v(0)\|^2+\|\sigma\|_{L_2(H,H)}^2t+2\int_0^t\langle v(r), \sigma dW(r)\rangle.
    \end{align*}
    Thus by the Burkholder-Davis-Gundy inequality we have 
    \begin{align*}
        \Eb\left(\int_0^T\|v(t)\|_1^2dr\right)^p\leq C(\|v(0)\|,\sigma,T)+\frac{1}{\nu^p}\Eb\sup_{t\in[0,T]}\left|\int_0^t\left\langle v(r), \sigma dW(r)\right\rangle\right|^p\leq C(p,\sigma,T,\|v_0\|),
    \end{align*}
    as desired. 

    Next we prove (ii). Recall the equation for $(\bv,\tu)$ as in \eqref{e.081001} and \eqref{e.080610}. Test the equation \eqref{e.080610} with $\phi$, we have 
    \begin{align*}
        \langle \bv(t), \phi\rangle& = \langle \bv(0), \phi\rangle -\int_0^t\left\langle\Pc_h \Big(\bv\cdot \nabla_h  \bv + e^{-2\aa Jr}\bP\left(\tu\cdot \nabla_h  \tu - \tu^{\perp}\cdot \nabla_h \tu^{\perp}\right) \Big), \phi \right\rangle dr\\
        &\quad + \int_0^t\langle \nu\DD_h \bv, \phi\rangle + \int_0^t\left\langle\phi, \Pc_h\overline{\sigma}dW(r)\right\rangle. 
    \end{align*}
    \cref{l.081001} with $d=2,m_1=0,m_2=0,m_3=2$ gives 
    \begin{align*}
        \left|\left\langle\Pc_h \Big(\bv\cdot \nabla_h  \bv\Big), \phi \right\rangle\right| =\left|\left\langle \bv\cdot \nabla_h  \bv, \Pc_h\overline{\phi} \right\rangle\right|\leq C\|\phi\|_2\|\bv\|\|\bv\|_1. 
    \end{align*}
    The same lemma with $d=3, m_1=0,m_2=0,m_3=2$ yields 
    \begin{align*}
        \left|\left\langle\Pc_h \Big(e^{-2\aa Jr}\bP\left(\tu\cdot \nabla_h  \tu - \tu^{\perp}\cdot \nabla_h \tu^{\perp}\right) \Big), \phi \right\rangle\right| &=\left|\left\langle\tu\cdot \nabla_h  \tu - \tu^{\perp}\cdot \nabla_h \tu^{\perp}, {\bP}e^{2\aa Jr}\Pc_h\phi \right\rangle\right|\\
        &\leq  C\|\phi\|_2\|\tu\|\|\tu\|_1\leq C\|\phi\|_2\|\tv\|\|\tv\|_1. 
    \end{align*}
    As a result, by Young's inequality and H\"older's inequality, 
    \begin{align*}
        \left|\langle \bv(t), \phi\rangle-\langle \bv(s), \phi\rangle\right|&-\left|\int_s^t\langle\phi, \Pc_h\overline{\sigma}dW(r)\rangle\right|\leq C\|\phi\|_2\int_s^t\|v(r)\|\left(\|v(r)\|_1+1\right)dr\\
        &\leq C\|\phi\|_2\left(\int_s^t\left(\|v(r)\|^3+1\right)dr+\int_0^t\|v(r)\|_1^{3/2}dr\right)\\
        &\leq C\|\phi\|_2\left(|t-s|\sup_{r\in[0,T]}\left(\|v(r)\|^3+1\right)+|t-s|^{\frac14}\left(\int_0^T\|v(r)\|_1^2dr\right)^{\frac34}\right). 
    \end{align*}
    It follows from the Burkholder-Davis-Gundy inequality that 
    \begin{align*}
        \Eb\left|\int_s^t\langle\phi, \Pc_h\overline{\sigma}dW(r)\rangle\right|^p\leq C_p\left(\int_s^t\|\phi\|^2\sum_n\|\sigma_n\|^2dr\right)^{p/2}\leq C|t-s|^{p/2}. 
    \end{align*}
    Then by \eqref{e.081002}, we obtain 
    \begin{align*}
        \Eb\left|\langle \bv(t), \phi\rangle-\langle \bv(s), \phi\rangle\right|^p\leq C |t-s|^q,
    \end{align*}
    with $q = p/4$. 

    Similarly, by testing the equation \eqref{e.081001} for $\tu$ with $\phi$ we have 
    \begin{align*}
        \langle \tu(t), \phi\rangle&=\langle \tu(0), \phi\rangle\\
        &\quad - \int_0^t\left\langle e^{\aa Jr}\left(\tv \cdot \nabla_h  \tv + \tv \cdot \nabla_h  \bv + \bv \cdot \nabla_h  \tv+w(\tv) \partial_z \tv - \bP\Big(\tv \cdot \nabla_h  \tv + (\nabla_h  \cdot \tv) \tv \Big)\right), \phi\right\rangle\\
        &\quad +\langle\nu\Delta\tu,\phi\rangle + \langle\phi,e^{\aa Jt}\widetilde{\sigma}dW(r)\rangle. 
    \end{align*}
    Here we only address the estimate for the term involving $w(\tv) \partial_z \tv$, since the remaining terms can be treated similarly as in estimating $\bv$. Note that 
    \begin{align*}
        w(\tv) \partial_z \tv = \partial_z(\tv w(\tv)) - \tv\partial_zw(\tv) = \partial_z(\tv w(\tv))+\tv\nabla_h \cdot \tv. 
    \end{align*}
    By integration by parts, we obtain 
    \begin{align*}
        \left|\left\langle e^{\aa Jr}w(\tv) \partial_z \tv, \phi\right\rangle\right|
        &\leq \left|\left\langle \tv w(\tv), e^{-\aa Jr}\pp_z\phi\right\rangle\right|+\left|\left\langle \tv(\nabla_h \cdot \tv), e^{-\aa Jr}\phi\right\rangle\right|\\
        &\leq C \|\phi\|_3\|\tv\|\|\tv\|_1. 
    \end{align*}
    Then by the same arguments as above, we obtain the desired inequality for $\tu$ with the same $q=p/4$. 

\end{proof}

\begin{lemma}[Estimates on the intermediate system \eqref{e.080803}]\label{l.092501}
    Let $v_0\in H^{2}, T>0, p\geq 2$ and $V^{\aa}=(\bV^{\aa},\tV^{\aa})$ be the solution to \eqref{e.080803} with initial data $v_0$. Then there is a constant $C=C(T,\|v_0\|_2,p)$ independent of $\aa$ such that 
    \[
        \Eb \left[\sup_{t\in[0,T]}\|V^{\aa}(t)\|_{2}^{p}+\left(\int_0^T\|V^{\aa}(t)\|_{3}^2dt\right)^p\right]\leq C.
    \]
\end{lemma}
\begin{proof}
    In the proof we drop the superscript $\aa$ for convenience. The proof is divided into the following four steps.
    
    \subsubsection*{Step 1: $L^2$ estimates}
    By It\^o formula and equation \eqref{e.080803} we have 
    \begin{align}\label{e.092403}
        \begin{split}
            \|\bV(t)\|^p &= \|\bV(0)\|^p-p\nu \int_0^t\|\nabla_h  \bV(r)\|^2\|\bV(r)\|^{p-2}dr+ \frac{p}{2}\|\mathcal P_h\overline{\ss}\|_{L_2(H,H)}^2\int_0^t\|\bV(r)\|^{p-2}dr\\
            &+\frac{p(p-2)}{2}\int_0^t\|\bV(r)\|^{p-4}\|\mathcal P_h\overline{\ss} \bV(r)\|^2 dr + p\int_0^t\|\bV(r)\|^{p-2}\left\langle \bV(r),\mathcal P_h\overline{\ss} dW(r)\right\rangle , \\
            \|\tV(t)\|^p& = \|\tV(0)\|^p-p\nu \int_0^t\|\nabla  \tV(r)\|^2\|\tV(r)\|^{p-2}dr +\frac{p}{2}\|\widetilde{\ss}\|_{L_2(H,H)}^2\int_0^t\|\tV(r)\|^{p-2}dr\\
            &+\frac{p(p-2)}{2}\sum_{k}\int_0^t\|\tV(r)\|^{p-4} \|e^{\aa Jr} \widetilde{\ss}\tV(r)\|^2 dr+ p\int_0^t\|\tV(r)\|^{p-2}\left\langle \tV(r),e^{\aa Jr} \widetilde{\ss} dW(r)\right\rangle ,
        \end{split}
    \end{align}
    where we have used the cancellation $\langle\bV\cdot\nabla_h  \bV,\bV \rangle = 0$ and $\langle\bV\cdot\nabla_h  \tV,\tV \rangle = 0$ since $\bV$ is divergence free, and the cancellation $\langle \tV^{\perp}\nabla_h ^{\perp}\cdot\bV, \tV\rangle $. By the same procedure as in deriving \eqref{e.081002}, we obtain from the equation \eqref{e.092403} that 
    \begin{align}\label{e.092402}
        \Eb\left[\sup_{t\in[0,T]}\|V(t)\|^{p}+\left(\int_0^T\|V(t)\|_1^2dt\right)^p\right]\leq C. 
    \end{align}

    \subsubsection*{Step 2: $H^1$ and $H^2$ estimates of $\bV$} Let $k\geq 1$. By It\^o's formula, from the equation of $\bV$ in \eqref{e.080803} we have 
    \begin{align*}
        \|\LL^k\bV(t)\|^p&= \|\LL^k\bV(0)\|^p+p \int_0^t\langle \LL^k(-\bV\cdot\nabla_h \bV+\nu\DD_h\bV), \LL^k\bV\rangle \|\LL^k\bV\|^{p-2}dr\\
        &+ \frac{p}{2}\|\mathcal P_h\overline{\ss}\|_{L_2(H,H^k)}^2\int_0^t\|\LL^k\bV\|^{p-2}dr+\frac{p(p-2)}{2}\int_0^t\|\LL^k\bV\|^{p-4}\|\mathcal P_h\overline{\ss}\LL^{2k}\bV\|^2dr\\
        & + p\int_0^t\|\LL^k\bV\|^{p-2}\left\langle \LL^k\bV,\LL^k\mathcal P_h\overline{\ss} dW(r)\right\rangle 
    \end{align*}
    For $k=1$ one has the cancellation $\langle\bV\cdot\nabla_h  \bV, \DD_h\bV\rangle=0$ on $\mathbb T^2$. Hence by the same proof as \eqref{e.081002}, one has 
    \begin{align}\label{e.092501}
        \Eb\left[\sup_{t\in[0,T]}\|\bV(t)\|_1^{p}+\left(\int_0^T\|\bV(t)\|_2^2dt\right)^p\right]\leq C. 
    \end{align}
    For $k=2$ there is no such cancellation. To deal with the nonlinear term, we apply \cref{l.080801} and divergence free property of $\bV$ to obtain 
    \begin{align*}
        |\langle \LL^2(\bV\cdot\nabla_h \bV),\LL^2\bV\rangle|&=|\langle \LL^2(\bV\cdot\nabla_h \bV)-\bV\cdot\nabla_h \LL^2\bV,\LL^2\bV\rangle|\\
        &\lesssim \|\LL^2\bV\|_{L^4}\|\nabla_h \bV\|_{L^{4}}\|\LL^2\bV\|\lesssim \|\bV\|_{1}^{3/2}\|\bV\|_{3}^{3/2}
    \end{align*}
    by interpolation inequalities. Therefore, by Young's inequality, we have 
    \begin{align*}
        &p \int_0^t\langle \LL^2(-\bV\cdot\nabla_h \bV+\nu\DD_h\bV), \LL^2\bV\rangle \|\LL^2\bV\|^{p-2}dr\\
        &\leq C\int_0^t\|\bV\|_{1}^{3/2}\|\bV\|_{3}^{3/2}\|\LL^2\bV\|^{p-2}dr-p\nu \int_0^t\|\LL^3\bV\|^2\|\LL^2\bV\|^{p-2}dr\\
        &\leq C\int_0^t\|\bV\|_{1}^{6}\|\LL^2\bV\|^{p-2}dr-\frac{p\nu}{2} \int_0^t\|\LL^3\bV\|^2\|\LL^2\bV\|^{p-2}dr\\
        &\leq \ee \sup_{t\in[0,T]}\|\LL^2\bV\|^{p}+C(\ee)\sup_{t\in[0,T]}\|\bV\|_{1}^{3p}-\frac{p\nu}{2} \int_0^t\|\LL^3\bV\|^2\|\LL^2\bV\|^{p-2}dr
    \end{align*}
    for any $\ee>0$. Combining this with the $H^1$ estimate \eqref{e.092501} of $\bV$ and applying Gr\"onwall's inequality as in deriving \eqref{e.081002}, we obtain 
    \begin{align}\label{e.092502}
        \Eb\left[\sup_{t\in[0,T]}\|\bV(t)\|_2^{p}+\left(\int_0^T\|\bV(t)\|_3^2dt\right)^p\right]\leq C. 
    \end{align}

    \subsubsection*{Step 3: $H^1$ estimate of $\tV$} 
    From equation \eqref{e.080803} for $\tV$ we have 
    \begin{align*}
        \|\LL\tV(t)\|^p &= \|\LL\tV(0)\|^p + p\int_0^t\|\LL\tV\|^{p-2}\langle - \bV\cdot\nabla_h \tV-\frac12\tV^{\perp}\nabla_h ^{\perp}\cdot\bV+\nu\DD\tV, \LL^2\tV\rangle dr\\
        & + \frac{p}{2}\|\widetilde{\ss}\|_{L_2(H,H^1)}^2\int_0^t\|\LL\tV\|^{p-2}dr + \frac{p(p-2)}{2}\int_0^t\|\LL\tV\|^{p-4}\|\widetilde{\ss}\LL^{2}\tV\|^2dr\\
        &+  p\int_0^t\|\LL\bV\|^{p-2}\left\langle \LL\tV,\LL e^{\aa Jr}\widetilde{\ss} dW(r)\right\rangle .
    \end{align*}
    Since $\bV$ is divergence free, by integration by parts we have 
    \begin{align*}
        \left|\langle\bV\cdot\nabla_h \tV, \LL^2\tV \rangle\right|&= \left|\langle \bV\cdot\nabla_h\tV,\Delta_h\tV\rangle+\langle \bV\cdot\nabla_h\tV,\partial_z^2\tV\rangle \right| = \left|\langle \bV\cdot\nabla_h\tV,\Delta_h\tV\rangle \right|
        \\
        &=\left|\langle\nabla_h \tV\nabla_h \bV, \nabla_h \tV \rangle\right|\leq \|\nabla_h \tV\|_{L^4}^2\|\nabla_h \bV\|\lesssim \|\tV\|_{7/4}^2\|\nabla_h \bV\|,
    \end{align*}
    where we used the Sobolev embedding $H^{3/4}\subset L^4$. By the interpolation $\|\tV\|_{7/4}\lesssim \|\tV\|^{\frac18}\|\tV\|_{2}^{\frac78}$ and Young's inequality, we deduce 
    \begin{align*}
        \left|\langle\bV\cdot\nabla_h \tV, \LL^2\tV \rangle\right|\leq C_{\ee} \|\tV\|^2\|\bV\|_1^8+\ee\|\tV\|_{2}^2
    \end{align*}
    for any $\ee>0$. 

    By integration by parts and orthogonality of $\tV^{\perp}$ with $\tV$, one has 
    \begin{align*}
        |\langle\tV^{\perp}\nabla_h ^{\perp}\cdot\bV, \LL^2\tV\rangle|& = |\langle\tV^{\perp}\nabla_h \left(\nabla_h ^{\perp}\cdot\bV\right), \nabla_h \tV\rangle|\\
        &\lesssim \|\tV\|_{L^4}\|\nabla_h \tV\|_{L^4}\|\bV\|_{2}\\
        &\lesssim \|\tV\|^{\frac14}\|\tV\|_{1}^{3/4}\|\tV\|_{1}^{1/4}\|\tV\|_{2}^{3/4}\|\bV\|_{2}\lesssim \|\tV\|^{\frac14}\|\tV\|_{2}^{7/4}\|\bV\|_{2}
    \end{align*}
    by Ladyzhenskaya's inequality in 3D. Applying Young's inequality we obtain 
    \begin{align*}
        |\langle\tV^{\perp}\nabla_h ^{\perp}\cdot\bV, \LL^2\tV\rangle|\lesssim C_{\ee} \|\tV\|^2\|\bV\|_2^8+\ee\|\tV\|_{2}^2. 
    \end{align*}
    Therefore, we have 
    \begin{align*}
        &p\int_0^t\|\LL\tV\|^{p-2}\langle - \bV\cdot\nabla_h \tV-\frac12\tV^{\perp}\nabla_h ^{\perp}\cdot\bV+\nu\DD\tV, \LL^2\tV\rangle dr\\
        &\leq C\int_0^t\|\LL\tV\|^{p-2}\|\tV\|^2\|\bV\|_2^8dr -\frac{p\nu}{2}\int_0^t\|\LL\tV\|^{p-2}\|\LL^2\tV\|^2dr\\
        &\leq C(\ee)\sup_{t\in[0,T]}\|\tV\|^p\|\bV\|_2^{4p} + \ee\sup_{t\in[0,T]}\|\LL\tV\|^{p} -\frac{p\nu}{2}\int_0^t\|\LL\tV\|^{p-2}\|\LL^2\tV\|^2dr 
    \end{align*}
    for any $\ee>0$. Combining this with the $H^2$ estimate \eqref{e.092502} of $\bV$ and $L^2$ estimate \eqref{e.092402} of $\tV$, and applying Gr\"onwall's inequality as in deriving \eqref{e.081002}, we obtain 
    \begin{align}\label{e.092503}
        \Eb\left[\sup_{t\in[0,T]}\|\tV(t)\|_1^{p}+\left(\int_0^T\|\tV(t)\|_2^2dt\right)^p\right]\leq C. 
    \end{align}

\subsubsection*{Step 4: $H^2$ estimate of $\tV$} Again from equation \eqref{e.080803} for $\tV$ we have 
    \begin{align}\label{e.092504}
       \begin{split}
            \|\LL^2\tV(t)\|^p &= \|\LL^2\tV(0)\|^p + p\int_0^t\|\LL^2\tV\|^{p-2}\langle - \bV\cdot\nabla_h \tV-\frac12\tV^{\perp}\nabla_h ^{\perp}\cdot\bV+\nu\DD\tV, \LL^4\tV\rangle dr\\
            &\quad + \frac{p}{2}\|\widetilde{\ss}\|_{L_2(H,H^1)}^2\int_0^t\|\LL^2\tV\|^{p-2}dr + \frac{p(p-2)}{2}\int_0^t\|\LL^2\tV\|^{p-4}\|\widetilde{\ss}\LL^4\tV\|^2 dr\\
            &\quad+  p\int_0^t\|\LL^2\bV\|^{p-2}\left\langle \LL^2\tV,\LL^2 e^{\aa Jr}\widetilde{\ss} dW(r)\right\rangle ,
       \end{split}
    \end{align}
    where we only need to deal with the terms involving $\bV$. By integration by parts, H\"older's inequality, Sobolev embedding $H^{3/4}\subset L^4$, interpolation followed by Young's inequality, we deduce 
    \begin{align*}
        |\langle\tV^{\perp}\nabla_h ^{\perp}\cdot\bV, \LL^4\tV\rangle|&\leq \|\nabla \tV\|_{L^4}\|\nabla_h \bV\|_{L^4}\|\tV\|_{3}+\|\tV\|_{L^{\infty}}\|\bV\|_{2}\|\tV\|_3\\
        &\lesssim \|\tV\|_{2}\|\bV\|_1^{1/2}\|\bV\|_2^{1/2}\|\tV\|_3+ \|\tV\|_{2}\|\bV\|_2\|\tV\|_3\\
        &\lesssim \|\tV\|_1^{1/2}\|\bV\|_2\|\tV\|_3^{3/2}\\
        &\leq C_{\ee}\|\tV\|_1^{2}\|\bV\|_2^4 + \ee \|\tV\|_3^{2}
    \end{align*}
    for any $\ee>0$. In addition, by integration by parts, H\"older's inequality, Ladyzhenskaya's inequality and interpolation followed by Young's inequality, one has  
    \begin{align*}
        \left|\langle\bV\cdot\nabla_h \tV, \LL^4\tV \rangle\right|&=\left|\langle \LL(\bV\cdot\nabla_h \tV), \LL^3\tV \rangle\right|\\
        &\lesssim\|\tV\|_3\left(\|\nabla_h \tV\|_{L^4}\|\LL\bV\|_{L^4}+\|\LL\nabla_h \tV\|_{L^4}\|\bV\|_{L^4}\right)\\
        &\lesssim \|\tV\|_{3}\left(\|\tV\|_1^{1/4}\|\tV\|_2^{3/4}\|\bV\|_2+ \|\tV\|_2^{1/4}\|\tV\|_3^{3/4}\|\bV\|_1\right)\\
        &\lesssim \|\tV\|_3^{7/4}\|\tV\|_1^{1/4}\|\bV\|_2+\|\tV\|_{1}^{1/8}\|\tV\|_{3}^{1/8}\|\tV\|_{3}^{7/4}\|\bV\|_1\\
        &\lesssim C_{\ee}\|\tV\|_1^2\|\bV\|_2^8+C_{\ee}\|\tV\|_1^2\|\bV\|_1^{16}+\ee\|\tV\|_3^2, 
    \end{align*}
    for any $\ee>0$. Consequently, 
    \begin{align*}
        &p\int_0^t\|\LL^2\tV\|^{p-2}\langle - \bV\cdot\nabla_h \tV-\frac12\tV^{\perp}\nabla_h ^{\perp}\cdot\bV+\nu\DD\tV, \LL^4\tV\rangle dr\\
        &\leq C(\ee)\sup_{t\in[0,T]}\|\tV\|_1^p\left(\|\bV\|_2^{2p}+\|\bV\|_2^{4p}+\|\bV\|_1^{8p}\right)+\ee\sup_{t\in [0,T]}\|\LL^2\tV\|^p - \nu\int_0^t\|\LL^2\tV\|^{p-2}\|\tV\|_3^2dr,
    \end{align*}
    which combined with \eqref{e.092501}, \eqref{e.092502}, \eqref{e.092503} and \eqref{e.092504} implies 
    \begin{align*}
        \Eb\left[\sup_{t\in[0,T]}\|\tV(t)\|_2^{p}+\left(\int_0^T\|\tV(t)\|_3^2dt\right)^p\right]\leq C, 
    \end{align*}
    by a Gr\"onwall inequality argument. 

\end{proof}

\begin{lemma}[Estimates on the limit resonant system]\label{l.120101}
    Let $V=(\bV,\tV)$ be a solution to the equation \eqref{e.081201} with initial data $V_0\in H$ and $T>0$. Then the following estimates hold. 
    \begin{enumerate}
        \item There exists $\eta_0>0$ and a constant $C=C(T)$ such that 
        \begin{align}\label{e.120106}
        \sup_{t\in[0,T]}\Eb\exp\left(\eta\nu\int_0^t\|\nabla  V(r)\|^2 dr\right)\leq C\exp(\eta\|V_0\|^2),
        \end{align}
        for all $\eta\in (0,\eta_0]$.
        \item For any $\eta>0$, there exists a constant $C=C(\eta,T)$ such that 
        \begin{align}\label{e.120107}
            \Eb\sup_{t\in[0,T]}\left(t^4\|V(t)\|_2^2+\int_0^tr^4\|V(r)\|_{3}^2dr\right)\leq C\exp(\eta\|V_0\|^2). 
        \end{align}
    \end{enumerate}  
\end{lemma}
\begin{proof}    
        The proof of  \eqref{e.120106}  is similar to the corresponding estimate for the 2D Navier-Stokes equation. Recall that 
        \[
            {\bP}G = \Pc_h\overline{\ss}, \quad (I-\bP)G = \sqrt{\widetilde{Q}}.
        \]
        Due to the cancellation property of the nonlinear terms, by It\^o's formula we have 
        \begin{align}\label{e.121507}
            \|V(t)\|^2 + 2\nu\int_0^t\|\nabla  V\|^2dr = \|V(0)\|^2 + t\|G\|_{L_{2}(H,H)}^2 + 2\int_0^t\langle V, GdW(r)\rangle,
        \end{align}
        Denote $M(t)= 2\int_0^t\langle V, GdW(r)\rangle$ the martingale term and $[M](t)$ its quadratic variation. Then we can rewrite the above equation as 
        \begin{align*}
            &\|V(t)\|^2 + \nu\int_0^t\|\nabla  V(r)\|^2dr - t\|G\|_{L_{2}(H,H)}^2\\
            &\quad = \|V(0)\|^2 - \nu\int_0^t\|\nabla  V(r)\|^2dr + \frac{\nu\gg}{2}[M](t) +  M(t)-\frac{\nu\gg}{2}[M](t),
        \end{align*}
        where $\gg = \frac{\lambda_1}{2\|G\|_{L_{2}(H,H)}^2}$, and $\lambda_1$ is the first eigenvalue of $-\Delta$ on $L^2(\mathbb T^3; \mathbb R^2)$. Note that by Poincar\'e's inequality, 
        \begin{align*}
            -\nu\int_0^t\|\nabla V(r)\|^2dr + \frac{\nu\gg}{2}[M](t)&\leq - \nu\int_0^t\|\nabla  V(r)\|^2dr + 4\|G\|_{L_{2}(H,H)}^2\frac{\nu\gg}{2}\int_0^t\|V(r)\|^2dr\leq 0. 
        \end{align*}
        Therefore, 
        \begin{align*}
            \|V(t)\|^2 + \nu\int_0^t\|\nabla  V(r)\|^2dr - t\|G\|_{L_{2}(H,H)}^2 -\|V(0)\|^2\leq  M(t)-\frac{\nu\gg}{2}[M](t).
        \end{align*}
        For any $p>1$, let $\eta_0=\frac{\nu\gg}{p}$. By the exponential martingale inequality, we derive
        \begin{align}\label{e.120108}
            \begin{split}
                &\Pb\left(\exp\left(\eta_0\|V(t)\|^2 + \nu\eta_0\int_0^t\|\nabla  V(r)\|^2dr - t\eta_0\|G\|_{L_{2}(H,H)}^2 -\eta_0\|V(0)\|^2\right)\geq e^{\frac{K}{p}}\right)\\
                &=\Pb\left(\|V(t)\|^2 + \nu\int_0^t\|\nabla V(r)\|^2dr - t\|G\|_{L_{2}(H,H)}^2 -\|V(0)\|^2\geq \frac{K}{\nu\gg}\right)\\
                &\leq \Pb\left(\sup_{t\geq 0 }\left(\nu\gg M(t)-\frac{(\nu\gg)^2}{2}[M](t)\right)\geq K\right)\leq e^{-K}.
            \end{split} 
        \end{align}
        Since for any non-negative random variable $X$ satisfying $\Pb(X\geq R)\leq R^{-p}$ for any $R\geq 1$, one has $\Eb X\leq C(p)$, by estimate \eqref{e.120108}, we derive 
        \begin{align*}
            \Eb\exp\left(\eta_0\|V(t)\|^2 + \nu\eta_0\int_0^t\|\nabla  V(r)\|^2dr - t\eta_0\|G\|_{L_{2}(H,H)}^2 -\eta_0\|V(0)\|^2\right)\leq C
        \end{align*}
        for any $t\geq 0$, which in turn implies \eqref{e.120106} by invoking Jensen's inequality.

        One can also prove that 
        \begin{align}\label{e.120308}
            \Eb\exp\left(\eta_0\|V(t)\|^2\right)\leq C\exp\left(\eta_0\|V_0\|^2\right)
        \end{align}
        for some constant $C>0$ independent of times $t$ and $T$, see for example \cite{kuksin2012mathematics}. 

        \vskip0.1in 
        Now we prove \eqref{e.120107}. Denote by
        \begin{align*}
            Y_{\phi}(k,t) =t^k\|\phi(t)\|_k^2 + \nu\int_0^tr^k\|\phi(r)\|_{k+1}^2dr.  
        \end{align*}
        Since the $\bV$ satisfies the 2D Navier-Stokes equation, the corresponding estimate follows from \cite[Proposition 2.4.12]{kuksin2012mathematics}: For any integer $k, m\geq 1$ and $T\geq 1$ there is a constant $C=C(k,m,T)>0$ such that 
        \begin{align}\label{e.120303}
            \Eb\sup_{t\in[0,T]}Y_{\bV}(k,t)^m\leq C(1+\|\bV(0)\|^{4m(k+1)}).
        \end{align}
        We also note that 
        \begin{align}\label{e.120304}
            \Eb\sup_{t\in[0,T]}Y_{V}(0,t)^m\leq C(1+\|V(0)\|^{4m})
        \end{align}
        by \eqref{e.120108} and \cite[Corollary 2.4.11]{kuksin2012mathematics}. 
    
        To estimate the $\tV$ component, applying It\^o's formula to the functional $F(t,\cdot) = t^{2k}\|\cdot\|_k^2$, we have for $k\geq 1$,
    \begin{align}\label{e.120301}
            \begin{split}
                t^{2k}\|\tV(t)\|_k^2 &= \int_0^t\left(2k r^{2k-1}\|\tV(r)\|_{k}^2 + 2r^{2k}\left\langle \LL^{2k}\tV, \nu\DD\tV-\bV\cdot\nabla_h \tV-\frac12\tV^{\perp}\left(\nabla_h ^{\perp}\cdot\bV\right)\right\rangle\right)dr\\
                & + \frac{t^{2k+1}}{2k+1}\|(I-\bP)G\|_{L_2(H,H^k)}^2+ 2\int_0^tr^{2k}\langle \LL^{2k}\tV, (I-\bP)GdW(r)\rangle. 
            \end{split}
        \end{align}
        We first estimate the case $k=1$. Let 
        \[M(t) = 2\int_0^tr^2\langle \LL^{2}\tV, (I-\bP)GdW(r)\rangle\]
        be the martingale term, whose quadratic variation satisfies 
        \begin{align*}
            [M](t)\leq 4\|(I-\bP)G\|_{L_2(H,H^1)}^2T^2\int_0^tr^{2}\|\tV(r)\|_{1}^2dr \leq 4\|(I-\bP)G\|_{L_2(H,H^1)}^2T^2\lambda\int_0^tr^{2}\|\LL^2\tV(r)\|^2dr, 
        \end{align*}
        for $t\in[0,T]$ and some $\lambda>0$ according to the Poincar\'e inequality. Letting $\gg = \frac{1}{4\|(I-\bP)G\|_{L_2(H,H^1)}^2T^2\lambda}$, one has 
        \begin{align}\label{e.120302}
            \frac{\nu\gg}{2}[M](t)\leq \frac{\nu}{2}\int_0^tr^{2}\|\LL^2\tV(r)\|^2dr.
        \end{align}
        
        For the nonlinear terms, we have 
        \begin{align*}
            \left|\langle\bV\cdot\nabla_h \tV, \LL^2\tV \rangle\right|
            &\leq \|\bV\|_{L^{6}}\|\nabla_h \tV\|_{L^3}\|\LL^2\tV\|\lesssim \|\bV\|_{\frac23}\|\|\nabla_h \tV\|^{\frac12}\|\LL^2\tV\|^{\frac32}\\
            &\leq C\|\bV\|_{\frac23}^{4}\|\nabla_h \tV\|^{2}+\frac{\nu}{8}\|\LL^2\tV\|^{2}\\
            &\leq C\|\bV\|^{\frac43}\|\bV\|_1^{\frac83}\|\nabla \tV\|^{2}+\frac{\nu}{8}\|\LL^2\tV\|^{2}
        \end{align*}
        due to the interpolation $\|u\|_{L^3}\lesssim \|u\|^{\frac12}\|u\|_1^{\frac12}$ in 3D and Sobolev embedding $H_{\frac23}\subset L^6$ in 2D, and 
        \begin{align*}
            |\langle\tV^{\perp}\nabla_h ^{\perp}\cdot\bV, \LL^2\tV\rangle|&\lesssim \|\tV\|_{L^4}\|\nabla_h  \bV\|_{L^4}\|\LL^2\tV\|\\
            &\lesssim\|\tV\|^{\frac14}\|\tV\|_{1}^{\frac34}\|\bV\|_{1}^{\frac12} \|\bV\|_{2}^{\frac12} \|\LL^2\tV\|\\
            &\leq C\|\tV\|^{\frac12}\|\bV\|_1\|\tV\|_1^{\frac32}\|\bV\|_2+\frac{\nu}{8}\|\LL^2\tV\|^{2}
        \end{align*}
        due to the interpolation $\|u\|_{L^4}\lesssim \|u\|^{\frac14}\|u\|_1^{\frac34}$ in 3D and $\|u\|_{L^4}\lesssim \|u\|^{\frac12}\|u\|_1^{\frac12}$ in 2D. Therefore, 
        \begin{align*}
            &\int_0^t2r^2\left\langle \LL^{2}\tV, -\bV\cdot\nabla_h \tV-\frac12\tV^{\perp}\left(\nabla_h ^{\perp}\cdot\bV\right)\right\rangle dr\\
            &\leq CT^{\frac23}\sup_{t\in[0,T]}\|\bV\|^{\frac43}(t\|\bV\|_1^2)^{\frac43}\int_0^t\|\nabla \tV\|^2dt + CT^{\frac12}\sup_{t\in[0,T]}\|\tV\|^{\frac12}(t^{\frac12}\|\bV\|_1)(t\|\bV\|_2)\int_0^t\|\tV\|_{1}^{\frac32} dr\\
            &\quad + \frac{\nu}{2}\int_0^tr^2\|\LL^2\tV\|^2dr \\
            &\leq CT^{\frac23}\sup_{t\in[0,T]}Y_{\bV}(0,t)^{\frac23}Y_{\bV}(1,t)^{\frac43}Y_V(0,t)+CT^{\frac34}\sup_{t\in[0,T]}Y_V(0,t)Y_{\bV}(1,t)^{\frac12}Y_{\bV}(2,t)^{\frac12}\\
            &\quad + \frac{\nu}{2}\int_0^tr^2\|\LL^2\tV\|^2dr.\\
        \end{align*}
        Combining this with \eqref{e.120301} and \eqref{e.120302}, we obtain
        \begin{align*}
            Z_{\tV}(1,t):=t^2\|\tV(t)\|_1^2+\nu\int_0^tr^2\|\LL^2\tV(r)\|^2dr\leq K_T + M(t) -\frac{\nu\gg}{2}[M](t),
        \end{align*}
        where 
        \begin{align*}
            K_T &= \frac{2T}{\nu}\sup_{t\in[0,T]}Y_{V}(0,t)+T^{3}\|G\|_{L_2(H,H^1)}^2\\
            &\quad +CT^{\frac23}\sup_{t\in[0,T]}Y_{\bV}(0,t)^{\frac23}Y_{\bV}(1,t)^{\frac43}Y_V(0,t)+CT^{\frac34}\sup_{t\in[0,T]}Y_V(0,t)Y_{\bV}(1,t)^{\frac12}Y_{\bV}(2,t)^{\frac12}.
        \end{align*}
        Therefore, by the exponential martingale inequality, we have 
        \begin{align}\label{e.120305}
            \Pb\left(\sup_{t\in[0,T]}Z_{\tV}(1,t) - K_T\geq R\right)\leq \Pb\left(\sup_{t\in[0,T]}\left(\nu\gg M(t) - \frac{(\nu\gg)^2}{2}[M](t)\right)\geq \nu\gg R\right)\leq e^{-\nu\gg R}.
        \end{align}
         Note that for non-negative random variables $a$ and $b$ and $m\geq 1$, one has 
        \begin{align*}
            \Eb a^m\leq 2^m\Eb(a-b)^m\mathbf{1}_{\{a>b\}} + 2^m\Eb b^m = 2^m\int_0^{\infty}\Pb(a-b>R^{1/m})dR + 2^m\Eb b^m. 
        \end{align*}
        Applying this inequality with \eqref{e.120305}, \eqref{e.120303}, and \eqref{e.120304} we have 
        \begin{align}\label{e.120306}
            \Eb \sup_{t\in[0,T]}Z_{\tV}(1,t)^m\leq CP(\|V_0\|),
        \end{align}
        where $P(\|V_0\|)$ is some polynomial in $\|V_0\|$.

        Now we estimate the $H^2$ norm of $\tV$. For the nonlinear terms we have 
        \begin{align*}
            \left|\langle\bV\cdot\nabla_h \tV, \LL^4\tV \rangle\right|&=\left|\langle \LL(\bV\cdot\nabla_h \tV), \LL^3\tV \rangle\right|\\
            &\lesssim \|\LL^3\tV\|\left(\|\LL\bV\|_{L^4}\|\nabla_h  \tV\|_{L^4}+\|\LL\nabla_h \tV\|\|\bV\|_{L^{\infty}}\right)\\
            &\lesssim \|\LL^3\tV\|\left(\|\bV\|_{2}\|\LL^2 \tV\|^{\frac34}\|\tV\|_1^{\frac14}+\|\LL^2\tV\|\|\bV\|_{2}\right)\\
            &\leq C\left(\|\bV\|_{2}^2\|\LL^2 \tV\|^{\frac32}\|\tV\|_1^{\frac12}+\|\LL^2\tV\|^2\|\bV\|_{2}^2\right) + \frac{\nu}{8}\|\tV\|_3^2, 
        \end{align*}
        and  similarly, 
        \begin{align*}
            |\langle\tV^{\perp}\nabla_h ^{\perp}\cdot\bV, \LL^4\tV\rangle| &= |\langle\LL(\tV^{\perp}\nabla_h ^{\perp}\cdot\bV), \LL^3\tV\rangle|\\
            &\lesssim \|\LL^3\tV\|\left(\|\LL\tV\|_{L^4}\|\nabla_h  \bV\|_{L^4}+\|\LL\nabla_h \bV\|\|\tV\|_{L^{\infty}}\right)\\
            &\lesssim \|\LL^3\tV\|\left(\|\bV\|_{2}\|\LL^2 \tV\|^{\frac34}\|\tV\|_1^{\frac14}+\|\bV\|_2\|\LL^2\tV\|\right)\\
            &\leq C\left(\|\bV\|_{2}^2\|\LL^2 \tV\|^{\frac32}\|\tV\|_1^{\frac12}+\|\LL^2\tV\|^2\|\bV\|_{2}^2\right) + \frac{\nu}{8}\|\tV\|_3^2.
        \end{align*}
    Then similar to the estimate for $k=1$, from the estimates for the nonlinear terms and \eqref{e.120301} we have
    \begin{align}\label{e.120307}
        Z_{\tV}(2,t): = t^4\|\tV\|_2^2+\nu\int_0^tr^4\|\LL^3\tV\|^2dr\leq L_T + N(t)-\frac{\nu\gg}{2}[N]_t, 
    \end{align}
    where 
    \[N(t) = 2\int_0^tr^{4}\langle \LL^{4}\tV, (I-\bP)GdW(r)\rangle, \quad \gg = \frac{1}{4\|(I-\bP)G\|_{L_2(H,H^2)}^2T^4\lambda}\]
    where the constant $\lambda$ is from the Poincar\'e inequality, and 
    \begin{align*}
        L_T &= C(T)\left(\sup_{t\in[0,T]}Z_{\tV}(1,t) + \|G\|_{L_2(H,H^2)}^2+\sup_{t\in[0,T]}Y_{\bV}(2,t)Z_{\tV}(1,t)\right).
    \end{align*}
    The desired inequality then follows from the exponential martingale inequality argument as in deriving \eqref{e.120306} and estimates \eqref{e.120303}, \eqref{e.120306} and \eqref{e.120307}. 
    
\end{proof}

\begin{lemma}[Higher order estimates]\label{l.080902}
    Assume $v$ is a strong solution to \eqref{PE-system} with deterministic initial data $v_0\in H^2$ and $\sigma\in L_2(H,H^2)$. Then for any $T>0$ and $R>e^e$, 
    \begin{align*}
        \Pb\left(\sup_{t\in[0,T]}\|v(t)\|_{2}^2+\nu\int_0^T\|v(t)\|_{2+1}^2dt\geq R\right)\leq \frac{C(T,\|v_0\|_{2},\sigma)}{\log\log\log R}.
    \end{align*}
\end{lemma}
\begin{proof}
    For $\kappa>0$ we denote by 
    \[
     \tau_\kappa:= \inf\limits_{T\geq 0} \left\{\int_0^T (\|\Lambda v\|^2 + \|\Lambda^2 v\|^2 + \|\Lambda v\|^2 \|\Lambda^2 v\|^2) dt > \kappa \right\}.
    \]
    Recall from \cite[Lemma 2.5]{glatt2014existence} that 
    \[
     \mathbb P(\tau_\kappa<T) \leq \frac{C(T,\|v_0\|_1,\sigma)}{\log\log \kappa}.
    \]

    Now we apply It\^o lemma for the $H^2$ estimate of \eqref{PE-system} and obtain
    \begin{align}\label{eqn:B3-1}
        d \|\Lambda^2 v\|^2 + 2\nu \|\Lambda^3 v\|^2 dt = -2 \langle v\cdot \nabla_h  v + w\partial_z v, \Lambda^4 v
        \rangle dt + \|\Lambda^2 \sigma\|^2 dt + 2\langle \Lambda^2 \sigma, \Lambda^2 v\rangle dW.
    \end{align}
    The nonlinear term can be bounded as
    \begin{align*}
        \left|-2 \langle v\cdot \nabla_h  v + w\partial_z v, \Lambda^4 v
        \rangle\right|
        &\leq C \|\Lambda^3 v\| \|\Lambda (v\cdot \nabla_h  v + w\partial_z v) \|
        \\
        &\leq C \|\Lambda^3 v\|(\|\Lambda v\|^{\frac12} \|\Lambda^2 v\|^{\frac32} + \|\Lambda v\|^{\frac12} \|\Lambda^2 v\| \|\Lambda^3 v\|^{\frac12})
        \\
        &\leq \nu \|\Lambda^3 v\|^2 + C\|\Lambda v\|^{2} \|\Lambda^2 v\|^4.
    \end{align*}
    As the noise is additive and regular enough, we can apply the Burkholder-Davis-Gundy inequality for any $0\leq \rho_a \leq \rho_b \leq T\wedge \tau_\kappa$ to obtain
    \begin{align*}
        &\mathbb E \sup\limits_{t\in[\rho_a,\rho_b]} \left|\int_{\rho_a}^{t} \langle \Lambda^2 \sigma, \Lambda^2 v\rangle dW_\xi \right|
        \leq C\mathbb E \left|\int_{\rho_a}^{\rho_b} \|\Lambda^2 v\|^2 d\xi \right|^{\frac12}
        \leq  \frac12 \mathbb E \sup\limits_{t\in[\rho_a,\rho_b]} \|\Lambda^2 v(t)\|^2 + C\mathbb E \int_{\rho_a}^{\rho_b} 1 d\xi.
    \end{align*}
    Therefore, integrate \eqref{eqn:B3-1} from $\rho_a$ to $t$, take a supremum over $[\rho_a,\rho_b]$, and then take expectation, we have
    \begin{align*}
        &\mathbb E \sup\limits_{t\in[\rho_a,\rho_b]} \|\Lambda^2 v(t)\|^2 + \nu\mathbb E\int_{\rho_a}^{\rho_b} \|\Lambda^3 v(t) \|^2 dt
        \\
        &\leq 2\mathbb E \|\Lambda^2 v(\rho_a)\|^2 + C\mathbb E \int_{\rho_a}^{\rho_b} (1+ \|\Lambda v(t)\|^{2} \|\Lambda^2 v(t)\|^2 \|\Lambda^2 v(t)\|^2) dt.
    \end{align*}
    Thanks to the definition of $\tau_\kappa$, using the Gr\"onwall inequality, we have
    \begin{align*}
        \mathbb E \sup\limits_{t\in[0,T\wedge \tau_\kappa]} \|\Lambda^2 v(t)\|^2 + \nu \mathbb E \int_0^{T\wedge \tau_\kappa} \|\Lambda^3 v(t)\|^2 dt \leq C\exp(C\kappa),
    \end{align*}
    where $C=C(T,\|v_0\|_2,\sigma)$. Now we use Markov's inequality to estimate
    \begin{align*}
        &\mathbb P\left(\sup_{t\in[0,T]}\|\Lambda^2 v(t)\|^2+\nu\int_0^T\|\Lambda^3 v(t)\|^2dt\geq R\right)
        \\
        &\leq  \frac{1}{R} \mathbb E\left(\sup_{t\in[0,T\wedge \tau_\kappa]}\|\Lambda^2 v(t)\|^2+\nu\int_0^{T\wedge \tau_\kappa}\|\Lambda^3 v(t)\|^2dt \right) + \mathbb P(\tau_\kappa < T)
        \\
        &\leq \frac{C\exp(C\kappa)}{R} + A,
    \end{align*}
    where 
    \[
      \mathbb{P}(\tau_\kappa < T)\leq A= \frac{C_1}{\log\kappa}(1+\gamma) + \frac{C_1}{\log\gamma}
    \]
    follows from \cite[Lemma 2.5]{glatt2014existence} and $C_1=C_1(T,\|v_0\|_1,\sigma)$ and $\gamma$ satisfies $\gamma>2+\|v_0\|_{L^6}^4$. Now we can choose for $R>0$ sufficiently large, $\kappa=\frac{\log(\sqrt{R})}C$, and $\gamma=\sqrt{\log\kappa}$, so that
    \begin{align*}
        &\mathbb P\left(\sup_{t\in[0,T]}\|\Lambda^2 v(t)\|^2+\nu\int_0^T\|\Lambda^3 v(t)\|^2dt\geq R\right)
        \\
        &\leq \frac{C}{\sqrt{R}}+ C_1\left(\frac1{\log\kappa} + \frac1{\sqrt{\log\kappa}} + \frac1{\log\log \kappa} \right) \leq \frac{C}{\log\log\log R},
    \end{align*}
    with $C=C(T,\|v_0\|_2,\sigma)$.

\end{proof}

\end{document}